\newcommand{\Aut}{\mathop{\mathrm{Aut}}\nolimits}
\newcommand{\De}{\frac{\mathrm{d}}{\mathrm{d}X}}
\newcommand{\Hom}{\mathop{\mathrm{Hom}}\nolimits}
\newcommand{\Ka}{K_{\mathrm{a}}}
\newcommand{\Kc}{K_{\mathrm{c}}}
\newcommand{\Kd}{K_{\mathrm{d}}}
\newcommand{\Maps}{\mathop{\mathrm{Maps}}\nolimits}
\newcommand{\prlim}{\mathop{\underleftarrow{\lim}}\limits}
\newcommand{\Sy}{\mathop{\mathfrak{S}}\nolimits}
\newcommand{\Sm}{\mathop{\mathrm{Sm}}\nolimits}

\documentclass[11pt,twoside]{amsart} 
\usepackage{amssymb,hyperref,pb-diagram}
\advance \topmargin by -\headheight
\advance \topmargin by -\headsep
\evensidemargin \oddsidemargin
\newtheorem{theorem}{Theorem}[section] 
\newtheorem{conjecture}[theorem]{Conjecture}

\newtheorem{lemma}[theorem]{Lemma}
\newtheorem{proposition}[theorem]{Proposition}
\theoremstyle{remark} 
\newtheorem{example}[theorem]{Example}
\theoremstyle{remark} 
\newtheorem{remark}[theorem]{Remark}
\begin{document} 
\title{L\"uroth's theorem for fields of rational functions in infinitely many permuted variables} 
\author{M.Rovinsky} 
\address{National Research University Higher School of Economics, 
AG Laboratory, HSE, 6 Usacheva str., Moscow, Russia, 119048}
\email{marat@mccme.ru}
\date{}

\begin{abstract} L\"uroth's theorem describes the dominant maps from rational curves over a field. 

In this note we study those dominant rational maps from cartesian powers $X^{\Psi}$ of geometrically 
irreducible varieties $X$ over a field $k$ for infinite sets $\Psi$ that are equivariant with 
respect to all permutations of the factors $X$. At least some of such maps arise as compositions 
$h:X^{\Psi}\xrightarrow{f^{\Psi}}Y^{\Psi}\to H\backslash Y^{\Psi}$, where $X\xrightarrow{f}Y$ is 
a dominant $k$-map and $H$ is a group of birational automorphisms of $Y|k$, acting diagonally on $Y^{\Psi}$. 

In characteristic 0, we show that this construction, when properly modified, gives all dominant 
equivariant maps from $X^{\Psi}$, if $\dim X=1$. For arbitrary $X$, the results are only partial. 

Also, a somewhat similar problem of describing the equivariant integral schemes over $X^{\Psi}$ 
of finite type is touched very briefly. 

 \end{abstract}

\maketitle 

\section{Introduction} 

L\"uroth's theorem is an explicit description \begin{itemize} \item of the intermediate fields between 
a field $k$ and the field $k(x)$ of rational functions over $k$ in one variable, or equivalently, 
\item of the dominant $k$-maps from rational curves over a field $k$. \end{itemize} 

Given a field extension $\Phi|k$ and a group $G$ of its automorphisms, one may similarly ask, 
whether there is a `simple' description of the intermediate $G$-invariant fields in $\Phi|k$. 
That would be an analogue of L\"uroth's theorem. 

The Galois correspondence gives such a description when $G$ is precompact and $k=\Phi^G$ 
is the fixed field of $G$. 

\vspace{2mm}

Consider the case of the symmetric group $G=\Sy_{\Psi}$ of all permutations of 
an infinite set $\Psi$ acting on fields $\Phi=F_{\Psi}$ of the following type. 
Let $F|k$ be a regular field extension of characteristic $p\ge 0$, and $F_{\Psi}=F_{k,\Psi}$ 
be the fraction field of the tensor product over $k$ of copies of $F$ labeled by the set $\Psi$. 

If $F$ is the function field of a geometrically irreducible scheme $X$ over $k$ then the intermediate 
$\Sy_{\Psi}$-invariant fields in $F_{\Psi}|k$ correspond to dominant rational maps from the cartesian 
power $X^{\Psi}$ over $k$ that are equivariant with respect to all permutations of the factors $X$. 

For each intermediate field $L$ in $F|k$ and a pro-algebraic $k$-group $H$ of field automorphisms of $L$ 
identical on $k$,\footnote{i.e. a functor $\{k\mbox{-algebras}\}\to\{\mbox{groups}\}$ of a specific type. 
In more details a structure of pro-algebraic group on field automorphism groups is discussed in 
\cite[\S1]{Demazure} (see also \cite{Hanamura,Montgomery}) and Appendix~\ref{group-actions-on-fields}.} 
the {\sl fixed subfield} $(L_{\Psi})^H$ in $L_{\Psi}$ of the diagonal $H$-action\footnote{defined 
in Appendix~\ref{group-actions-on-fields}, but for reduced $H$: $(L_{\Psi})^H=
((L_{\Psi}\otimes_kk^{\mathrm{sep}})^{H(k^{\mathrm{sep}})})^{\mathrm{Gal}(k^{\mathrm{sep}}|k)}$ 
for a separable closure $k^{\mathrm{sep}}$ of $k$.} is evidently invariant under the natural action of 
$\Sy_{\Psi}$ on $F_{\Psi}$, while $\mathrm{tr.deg}(L_{\Psi}|(L_{\Psi})^H)$ coincides with dimension of $H$. 
Here $\mathrm{tr.deg}$ denotes the transcendence degree. 

Let $K$ be an $\Sy_{\Psi}$-invariant subfield of $F_{\Psi}|k$. Set $d:=\mathrm{tr.deg}(F_{\Psi}|K)$. 
There are evidences that $K$ comes from the above construction: 
\begin{itemize} \item (Theorem~\ref{subfields-of-F-Psi-char-0}) if $p=0$ then $K$ is 
`of finite codimension' in $L_{\Psi}$ for a field extension $L|k$ in $F$, in the sense that 
$K\subseteq L_{\Psi}$ and for any $L'|k$ in $L$ of finite transcendence degree there is 
a field extension $L''|L'$ in $L$ with $\mathrm{tr.deg}(L''_{\Psi}|K\cap L''_{\Psi})<\infty$; 
\item if $\mathrm{tr.deg}(F|k)=1$ and $K\neq k$ then $d$ 
is finite (moreover, $d\le 3$ if $p=0$, see Theorem~\ref{invar-subf}); 
\item (Theorem~\ref{invar-subf}, Propositions~\ref{sub-algebraic}, \ref{alg-sub-KPsi}) 
$K$ is obtained by this construction if $F|k$ is of transcendence degree 1 and $p=0$; 
\item (Proposition~\ref{sub-algebraic}) if $F_{\Psi}$ is algebraic over $K$ then 
$F'_{\Psi}\subseteq K$ for an intermediate subfield $F'$ in $F|k$ over which $F$ is algebraic; 
if, moreover, $p=0$ then $K$ is precisely of type $(L_{\Psi})^H$. \end{itemize} 
In the case of $\mathrm{tr.deg}(F|k)=1$ and $p=0$, the $\Sy_{\Psi}$-invariant subfields of $F_{\Psi}$ 
admit a concrete description related to (certain systems of isogenies of) one-dimensional algebraic $k$-groups. 

Our principal tools are (a) the K\"ahler differentials considered as objects of the category 
$\Sm_{F_{\Psi}}(\Sy_{\Psi})$ of {\sl smooth} semilinear representations of $\Sy_{\Psi}$ over $F_{\Psi}$, 
and (b) a description from \cite{H90} (slightly generalized in Appendix \ref{Addendum-to-H90}) 
of the indecomposable injectives of the category $\Sm_{F_{\Psi}}(\Sy_{\Psi})$. 

In somewhat opposite direction, for a geometrically irreducible $k$-scheme $X$, 
one may look for extensions of the $\Sy_{\Psi}$-action on $X^{\Psi}$ to 
$\Sy_{\Psi}$-actions on generic points of schemes of finite type over 
$X^{\Psi}$. Such schemes are obtained by composing a) inclusions into $X^{\Psi}$ of $\Sy_{\Psi}$-invariant 
irreducible $k$-subschemes of $X^{\Psi}$, and b) dominant rational $\Sy_{\Psi}$-equivariant $k$-maps of 
finite type to invariant $k$-subschemes of $X^{\Psi}$. 

It is shown in \S\ref{Some-fin-gen-field-exts} that, for any $k$-variety $Y$, any smooth 
$\Sy_{\Psi}$-action on the function field $F_{\Psi}(Y)$ extending the natural $\Sy_{\Psi}$-action 
on $F_{\Psi}$ comes from an isomorphism $F_{\Psi}(Y)\cong F_{\Psi}(Y')$ for some $k$-variety $Y'$ 
and the trivial $\Sy_{\Psi}$-action on $k(Y')$. It is also shown that for certain classes of 
$F_{\Psi}$-varieties $\mathbb W$ the existence of a smooth $\Sy_{\Psi}$-action on $F_{\Psi}(\mathbb W)$ 
implies that $\mathbb W$ is birational to $W\times_kF_{\Psi}$ for some $k$-variety $W$. 

For each dominant $k$-map $X\xrightarrow{f}Y$ with geometrically irreducible general fibre, 
an invariant $k$-subscheme of $X^{\Psi}$ can be obtained as the pullback under $f^{\Psi}$ of 
a subscheme of $Y$ embedded diagonally into $Y^{\Psi}$. It follows from \cite[Proposition 2.7]{NS} 
that all integral invariant $k$-subschemes of $X^{\Psi}$ arise this way if $\dim X=1$. 
Another proof of this corollary is contained in \S\ref{subvar-of-curve-powers}. 

Though the principal results here deal with the symmetric group $\Sy_{\Psi}$, some assertions are stated for 
more general `permutation' groups.\footnote{A {\sl permutation group} is a Hausdorff group $G$ admitting 
a base $B$ of neighbourhoods of 1 consisting of subgroups. In fact, $G$ is the {\sl permutation group of 
the set} $\coprod_{U\in B}G/U$.} A group that is the most interesting from the point of view of geometric 
applications is mentioned in Example \ref{ex-triv-fdim}. In more detail, this is discussed in \cite{rep,H90}.

\subsection{Notation} \label{NotaF} 
For a unital associative ring $A$ and a set $S$, both endowed with an action of 
a group $G$, $A\langle S\rangle=\{\sum_ia_i[s_i]~|~a_i\in A,\ s_i\in S\}$ denotes the free left 
$A$-module with basis $S$ (and the evidently defined left $A$-multiplication and addition) with the 
(diagonal) left $G$-action both on $A$ and $S$: $g(a[s])=a^g[gs]$ for all $g\in G$, $a\in A$, $s\in S$, 
where we write $a^g$ for the result of applying of $g$ to $a$. Then (for $S=G$) $A\langle G\rangle$ 
becomes a unital associative ring, while $A\langle S\rangle$ becomes a left $A\langle G\rangle$-module. 

For each pair of sets $\Upsilon\subseteq S$, denote (i) by $\Sy_S$ the group of 
all permutations of the set $S$, (ii) by $\Sy_{S|\Upsilon}$ the pointwise stabilizer of $\Upsilon$ in the 
group $\Sy_S$. Denote by $S^G:=\{s\in S\ |\ s^g=s\mbox{ for all }g\in G\}$ the subset of $S$ fixed by $G$. 

The notation $F|k$ is reserved for regular field extensions, i.e. such that the ring $F\otimes_kK$ 
is integral for any field extension $K|k$. 

Denote by $F_S=F_{k,S}$ the fraction field of the coproduct $\bigotimes_{k,~i\in S}F$ in 
the category of commutative $k$-algebras of the collection of copies of $F$ indexed by the set $S$. 
The symmetric group $\Sy_S$ of $S$ acts naturally on $F_S$. In particular, if $T$ is a variable then 
$k(S):=k(T)_S$ is the field of rational functions over $k$ in the variables labeled by the set $S$. 

For each field extension $L|K$, we denote by $\Omega_{L|K}$, $\Aut_{\mathrm{field}}(L|K)$ and 
$\mathrm{tr.deg}(L|K)$, respectively, the $L$-vector space of K\"ahler differentials on $L$ 
over $K$, the group of $K$-linear field automorphisms of $L$, and the transcendence degree.

\section{Relative dimension over invariant subfields of \texorpdfstring{$F_{\Psi}$}{}} 
Let $\Psi$ be a set, $G$ be a permutation group of $\Psi$. A $G$-action on a set $S$ is called {\sl smooth} 
(and the set $S$ endowed with such an action is called a {\sl smooth $G$-set}) if the stabilizer in $G$ of 
each element of the set $S$ is {\sl open}, i.e. contains the pointwise stabilizer of a finite subset of $\Psi$. 

For a ring $A$ and a permutation group $G$ acting on $A$, denote by $\Sm_A(G)$ 
the category of smooth left $A\langle G\rangle$-modules. This is a {\sl Grothendieck} category. 

Recall that an object $U$ in a category is called a {\sl cogenerator} if any pair 
of distinct morphisms $g_1,g_2\colon X\rightrightarrows Y$ admits a morphism 
$\theta\colon Y\to U$ such that $\theta\circ g_1\neq\theta\circ g_2$. 
For a category with direct products, an object $U$ is a cogenerator if and 
only if any object is {\it a subobject of a direct product} of copies of $U$. 

\begin{lemma} \label{level-1_quotients} Let $\Psi$ be an infinite set, $K$ be a smooth 
$\Sy_{\Psi}$-field, $P:=K\langle\Psi\rangle$, and $V$ be a $K\langle\Sy_{\Psi}\rangle$-submodule of a 
direct sum $M$ of copies of $P$ and of $K$. Suppose that $K$ is a cogenerator of $\Sm_K(\Sy_{\Psi})$, 
and $P$ is injective. Then $M/V$ is also a direct sum of copies of $P$ and of $K$. \end{lemma} 
\begin{proof} Let $\alpha=\sum\limits_{i=0}^ea_i[u_i]\in P$ be a non-zero element for some 
$e\ge 1$, $a_i\in K^{\times}$ and pairwise distinct $u_i\in\Psi$. Then any element of $P$, presented as 
$\sum\limits_{j=1}^Nb_j[v_j]+\sum\limits_{i=1}^ec_i[u_i]$, where $v_j\notin\{u_1,\dots,u_e\}$ and 
$b_j,c_i\in K$, coincides with $\sum\limits_{j=1}^Nb_j(a_0^{-1}\alpha)^{g_j}+\sum\limits_{i=1}^ed_i[u_i]$ 
for arbitrary $g_j\in\Sy_{\Psi|\{u_1,\dots,u_e\}}$ with $u_0^{g_j}=v_j$, and some $d_i\in K$. This means 
that $\alpha$ generates a $K\langle\Sy_{\Psi}\rangle$-submodule in $P$ of $K$-codimension $\le e$. As 
$K$ is a cogenerator of $\Sm_K(\Sy_{\Psi})$, this shows that any quotient of $P$ by a non-zero submodule 
is a finite direct sum of copies of $K$. 

As $M$ is injective,\footnote{Clearly, a simple cogenerator is injective. By \cite[Theorem 3.18]{H90}, 
the category $\Sm_K(\Sy_{\Psi})$ is locally noetherian, so any direct sum of injectives is injective.} 
any maximal essential extension $\widetilde{V}$ of $V$ in $M$ is injective 
as well,\footnote{In fact, $\widetilde{V}$ is the common kernel of the annihilator of $V$ in 
$\mathrm{End}_{K\langle\Sy_{\Psi}\rangle}(M)$.} and therefore, $\widetilde{V}$ splits off $M$ as 
a direct summand: $M=M'\oplus\widetilde{V}$. By Krull--Remak--Schmidt--Azumaya theorem (applicable, 
since $\mathrm{End}_{K\langle\Sy_{\Psi}\rangle}(K)=K^{\Sy_{\Psi}}$ and 
$\mathrm{End}_{K\langle\Sy_{\Psi}\rangle}(P)\cong K^{\Sy_{\Psi|\{x\}}}$ are fields for any 
$x\in\Psi$), both, $\widetilde{V}$ and $M'$, are direct sums of copies of $P$ and of $K$. 

As $V$ is essential in $\widetilde{V}$, the restriction of the projection $\widetilde{V}\to\widetilde{V}/V$ 
to each of the direct summands $P$ and $K$ is not injective, which means that $\widetilde{V}/V$ is a sum 
of quotients of $P$ by non-zero submodules, i.e. $\widetilde{V}/V$ is a sum of copies of $K$. \end{proof}

\begin{conjecture} \label{level-l-fields} Let $\Psi$ be an infinite set, $F|k$ be a non-trivial 
regular field extension. Then, for any $\Sy_{\Psi}$-invariant field extension $K|k$ in $F_{\Psi}$, 
there is a unique field extension $L|k$ in $F$ such that {\rm (i)} $K\subseteq L_{\Psi}$, {\rm (ii)} $L$ 
is algebraically closed in $F$, {\rm (iii)} for any $L'|k$ in $L$ with $\mathrm{tr.deg}(L'|k)<\infty$ 
there is a field extension $L''|L'$ in $L$ with $\mathrm{tr.deg}(L''_{\Psi}|K\cap L''_{\Psi})<\infty$. 
\end{conjecture}

The following lemma is definitely well-known. 
\begin{lemma} \label{horiz-functns} Let $F|k$ be a characteristic $0$ regular field extension, 
$K|k$ be a field extension, $\widetilde{F}$ be the fraction field of $F\otimes_kK$, and 
$D\in\mathrm{Der}_k(F)\subseteq\mathrm{Der}_K(\widetilde{F})$. Then 
$\widetilde{F}^{D=0}:=\{f\in\widetilde{F}~|~Df=0\}$ is the fraction field of $L\otimes_kK$, 
where $L=\{f\in F~|~Df=0\}$. \end{lemma} 
\begin{proof} Let $\xi=\alpha/\beta\in\widetilde{F}^{D=0}$, where $\alpha=\sum_ia_i\otimes b_i$ 
and $\beta=\sum_jc_j\otimes d_j\neq 0$ for some $a_i,c_j\in F$, $b_i,d_j\in K$. Let $B$ be 
the $k$-subalgebra of $K$, generated by $b_i,d_j\in K$. An arbitrary $k$-algebra homomorphism 
$B\xrightarrow{\varphi}C$ induces an $F$-algebra homomorphism 
$F\otimes_kB\xrightarrow{\varphi_F}F\otimes_kC$ and, if $\varphi_F(\beta)\neq 0$, --- an $F$-algebra 
homomorphism $(F\otimes_kB)[\beta^{-1}]\xrightarrow{\varphi_F}(F\otimes_kC)[\varphi_F(\beta)^{-1}]$. 

On the other hand, the $C$-linear extension of $D$ induces a derivation 
$D_C\in\mathrm{Der}_C((F\otimes_kC)[\varphi_F(\beta)^{-1}])$. Clearly, 
$D_C(\varphi_F(\xi))=0\in(F\otimes_kC)[\varphi_F(\beta)^{-1}]$.\footnote{since the following 
diagram commutes \[\begin{diagram} \node{F\otimes_kB}\arrow{e,t}{id_F\otimes\varphi} 
\arrow{s,l}{D_B:=D\otimes id_B} \node{F\otimes_kC} \arrow{s,r}{D_C:=D\otimes id_C}\\ 
\node{F\otimes_kB}\arrow{e,t}{id_F\otimes\varphi} \node{F\otimes_kC.}\end{diagram}\]} 
If $K|k$ is finite then $\widetilde{F}=F\otimes_kK$, hence $\widetilde{F}^{D=0}=L\otimes_kK$. 
Therefore, $\varphi(\xi)\in L\otimes_kC$ if, for an arbitrary maximal ideal $\mathfrak{m}$ of $B[\beta^{-1}]$, 
$\varphi$ is the natural homomorphism to the finite extension $C=C_{\mathfrak{m}}:=B[\beta^{-1}]/\mathfrak{m}$ 
of $k$. This means that the fibres over the closed points 
$\mathrm{Spec}(L\otimes_kC_{\mathfrak{m}})\xrightarrow{\varphi^*}Y$ of the morphism 
$\mathrm{Spec}((L\otimes_kB)[\xi])\xrightarrow{\eta}Y:=\mathrm{Spec}(L\otimes_kB)$ 
are single points. The set of such points is dense, so by Zariski's main theorem in the form of 
\cite[Th\'eor\`eme (4.4.3)]{EGAIII},\footnote{Let $Y$ be a noetherian scheme, $f:X\to Y$ be 
a quasi-projective morphism, $X'$ be set of points $x\in X$ that are isolated in their fibre $f^{-1}(f(x))$. 
Then $X'$ is open in $X$, and the induced scheme of $X'$ is isomorphic to an open part of a scheme that is 
finite over $Y$.} $\eta$ is birational. In other words, $\xi$ belongs to the fraction fields of 
$L\otimes_kB\subset L\otimes_kK$. \end{proof}

\label{tensor-factors} 
For each $u\in\Psi$, let $(u)\colon F\hookrightarrow F_{\Psi}$ be the field embedding 
identifying $F$ with the $u$-th tensor factor in $\bigotimes_{k,\Psi}F$. For each 
$f\in F$, let $f(u)$ be the image of $f$ under $(u)$. Then $(u)$ induces an inclusion 
$(u)\colon\Omega_{F|k}\hookrightarrow\Omega_{F_{\Psi}|k}$. Set $F_u:=\{f(u)~|~f\in F\}$. 

Let us check Conjecture~\ref{level-l-fields} in characteristic $0$. 
\begin{theorem} \label{subfields-of-F-Psi-char-0} Let $F|k$ be a regular field extension 
of characteristic $0$, and $K|k$ be an $\Sy_{\Psi}$-invariant field extension in 
$F_{\Psi}$. Then there is a unique field extension $L|k$ in $F$ such that {\rm (i)} 
$K\subseteq L_{\Psi}$, {\rm (ii)} $L$ is algebraically closed in $F$, {\rm (iii)} for any 
$L'|k$ in $L$ with $\mathrm{tr.deg}(L'|k)<\infty$ there is a field extension $L''|L'$ 
in $L$ with $\mathrm{tr.deg}(L''_{\Psi}|K\cap L''_{\Psi})<\infty$. \end{theorem}
\begin{proof} The object $F_{\Psi}\langle\Psi\rangle$ is a right $F$-vector space under 
$(\sum_if_i[u_i])\cdot a:=\sum_if_ia(u_i)[u_i]$ for all $a\in F$. Then the natural map 
$F_{\Psi}\langle\Psi\rangle\otimes_F\Omega_{F|k}\to\Omega_{F_{\Psi}|k}$, 
$f[u]\otimes\omega\mapsto f\omega(u)$, is bijective. 

By Theorem~\ref{H90Theorem1.2}, $F_{\Psi}$ is a cogenerator of $\Sm_{F_{\Psi}}(\Sy_{\Psi})$, and 
$F_{\Psi}\langle\Psi\rangle$ is injective. Then, by Lemma~\ref{level-1_quotients}, the quotient 
$\Omega_{F_{\Psi}|K}$ of $\Omega_{F_{\Psi}|k}\cong F_{\Psi}\langle\Psi\rangle\otimes_F\Omega_{F|k}$ 
is isomorphic to a direct sum of copies of $F_{\Psi}\langle\Psi\rangle$ and of $F_{\Psi}$. 

The tensor-hom adjunction induces a natural isomorphism of $F$-vector spaces \[\mathrm{Der}_k(F)
\xrightarrow{\sim}\Hom_{F_{\Psi}\langle\Sy_{\Psi}\rangle}(\Omega_{F_{\Psi}|k},F_{\Psi}\langle\Psi\rangle),
\mbox{ given directly by }D\mapsto\left[\omega\mapsto\sum_{u\in\Psi}\langle D_u,\omega\rangle[u]\right],\] 
with $D_u$ corresponding to $D$ under the natural embedding 
$\mathrm{Der}_k(F)\hookrightarrow\mathrm{Der}_{F_{\Psi\smallsetminus\{u\}}}(F_{\Psi})$. 

Let $S:=\Hom_{F_{\Psi}\langle\Sy_{\Psi}\rangle}(\Omega_{F_{\Psi}|K},F_{\Psi}\langle\Psi\rangle)
\subseteq\mathrm{Der}_k(F)$ be the set of derivations `vanishing'\ on $K$, i.e. 
\[S=\{D\in\mathrm{Der}_k(F)~|~D_uf=0\mbox{ for all }f\in K\mbox{ and }u\in\Psi\},\] and 
$M=\{\eta\in\Omega_{F_{\Psi}|k}~|~\langle D_u,\eta\rangle=0\mbox{ for all }D\in S\mbox{ and }u\in\Psi\}$ 
be the common `kernel'\ of $S$. Then $S$ is an $F$-vector subspace, 
$K\subseteq\widetilde{K}:=\{f\in F_{\Psi}~|~\mathrm{d}f\in M\}$, and 
$M/(F_{\Psi}\otimes_K\Omega_{K|k})$ is isomorphic to a direct sum of copies of $F_{\Psi}$. 

For each $u\in\Psi$, $F_{\Psi}$ can be considered as the fraction field 
of $F_{\{u\}}\otimes_kF_{\Psi\smallsetminus\{u\}}$. 

Then, by Lemma~\ref{horiz-functns}, $\{f\in F_{\Psi}~|~D_uf=0\mbox{ for all }D\in S\}$ 
is the fraction field of $L_{\{u\}}\otimes_kF_{\Psi\smallsetminus\{u\}}$, where 
$L:=\{f\in F~|~Df=0\mbox{ for all }D\in S\}$. As $u\in\Psi$ is arbitrary, we get 
$\widetilde{K}=L_{\Psi}$. 

For any $L'|k$ in $L$ with $\mathrm{tr.deg}(L'|k)<\infty$, the object 
$L_{\Psi}\otimes_{L'_{\Psi}}\Omega_{L'_{\Psi}|k}$ of $\Sm_{L_{\Psi}}(\Sy_{\Psi})$ is noetherian, 
so its image $(L_{\Psi}\otimes_{L'_{\Psi}}\Omega_{L'_{\Psi}|k})/(L_{\Psi}\otimes_{L'_{\Psi}}
\Omega_{L'_{\Psi}|k}\cap L_{\Psi}\otimes_K\Omega_{K|k})$ in $\Omega_{L_{\Psi}|K}$ is finite-dimensional. 
Replacing $L'$ by an appropriate extension $L''|L'$ in $L$ with $\mathrm{tr.deg}(L''|k)<\infty$, 
we can identify the space 
$(L_{\Psi}\otimes_{L''_{\Psi}}\Omega_{L''_{\Psi}|k})/(L_{\Psi}\otimes_{L''_{\Psi}}\Omega_{L''_{\Psi}|k}\cap 
L_{\Psi}\otimes_K\Omega_{K|k})$ with $L_{\Psi}\otimes_{L''_{\Psi}}\Omega_{L''_{\Psi}|L''_{\Psi}\cap K}$, 
which means that $\mathrm{tr.deg}(L''_{\Psi}|K\cap L''_{\Psi})<\infty$. \end{proof} 

\begin{remark} \label{unbounded_tr_deg} Though the transcendence degree of $L''_{\Psi}|K\cap L''_{\Psi}$ in 
Theorem~\ref{subfields-of-F-Psi-char-0} is finite, it cannot be bounded if $\mathrm{tr.deg}(F|k)>1$. 

E.g., take $F=k(A,B)$ and, for each integer $n\ge 0$, the $(n+1)$-dimensional unipotent $k$-subgroup 
\[G_n=\{\varphi(A,B)\mapsto\varphi(A+a,B+P(A))~|~a\in k,\ P\in k[T],\ \deg P<n\}\cong
\mathbb G_{\mathrm{a},k}\ltimes\mathbb G_{\mathrm{a},k}^n\mbox{ of }\Aut_{\mathrm{field}}(F|k).\] 

Fix some pairwise distinct $x_1,\dots,x_n\in\Psi$. Then 
\[F_{\Psi}=k(A_{x_1},A'_u;B_{x_1},B_{x_2}^{(1)},\dots,B_{x_n}^{(n-1)},B_v^{(n)}~|~u\in\Psi\smallsetminus\{x_1\},\ 
v\in\Psi\smallsetminus\{x_1,\dots,x_n\}),\] 
where $A'_u:=A_u-A_{x_1}$, $B_u^{(s)}:=\frac{B_u^{(s-1)}-B_{x_s}^{(s-1)}}{A_u-A_{x_s}}$, $B_u^{(0)}:=B_u$, 
while the elements $A'_u$ and $B_v^{(n)}$ are fixed by $G_n$, so (if $\#k=\infty$ or, as usual, 
if $G_n$ is considered as an algebraic group) 
\[(F_{\Psi})^{G_n}=k(A'_u;B_v^{(n)}~|~u\in\Psi\smallsetminus\{x_1\},\ v\in\Psi\smallsetminus\{x_1,\dots,x_n\})\] 
and thus, $\mathrm{tr.deg}(F_{\Psi}|(F_{\Psi})^{G_n})=\dim G_n=n+1$ is finite, but unbounded. \end{remark}

\section{Invariant subfields of \texorpdfstring{$F_{\Psi}$}{} in the case of 
one-dimensional \texorpdfstring{$F|k$}{F|k} of characteristic \texorpdfstring{$0$}{0}} 

\subsection{Finite-dimensional Lie subalgebras of the derivation algebra of 
a one-dimensional field extension of characteristic \texorpdfstring{$0$}{0}} 
Let $F|k$ be a regular field extension of characteristic $0$. Assume that $F|k$ is of transcendence 
degree 1 (i.e., $F|k(X)$ is algebraic for some $X\in F\smallsetminus k$). Then the algebra 
$\mathrm{Der}_k(F)$ of $k$-linear derivations of $F$ is a one-dimensional left $F$-vector space. 
\begin{lemma} \label{vanishing-p-powers} The non-zero abelian Lie $k$-subalgebras of $\mathrm{Der}_k(F)$ 
are one-dimensional. \end{lemma} 
\begin{proof} Fix some $X\in F$ such that $F|k(X)$ is an algebraic separable field extension. If $[D_1,D_2]=0$ 
then $\De(f_1/f_2)=0$, where $D_i=f_i\De$ for some $f_i\in F^{\times}$, i.e. $D_2\in k\cdot D_1$. \end{proof}

\begin{proposition} \label{restr-fin-dim} Let $\mathcal L\subset\mathrm{Der}_k(F)$ be a finite-dimensional 
Lie $k$-subalgebra, and $d:=\dim_k\mathcal L$. Then $d\le 3$, and there exists $X\in F\smallsetminus k$ 
such that \begin{enumerate} \item $\mathcal L=k\De\oplus kX\De\oplus kX^2\De$ {\bf if} $d=3;$ 
\item $\mathcal L=k\De\oplus kX\De$ {\bf if} $d=2$. \end{enumerate} \end{proposition} 
\begin{proof} If $d=2$ then let $\eta_1\in[\mathcal L,\mathcal L]\smallsetminus\{0\}$, and $\eta_2\in\mathcal L$ 
be such that $[\eta_1,\eta_2]=\eta_1$, i.e. $\eta_1(\eta_2/\eta_1)=1$. Set $R:=\eta_2/\eta_1\in F\smallsetminus k$, 
so $\eta_1=\frac{\mathrm{d}}{\mathrm{d}R}$ and $\eta_2=R\frac{\mathrm{d}}{\mathrm{d}R}$. 
As $\eta_1$ is defined uniquely up to a $k^{\times}$-multiple, and $\eta_2$ is defined uniquely 
modulo $[\mathcal L,\mathcal L]$, the class of $R$ in $\mathbb P_k(F/k)$ is well-defined. 

Let us show that $d\le 3$. Fix a rank one valuation $v\colon F^{\times}/k^{\times}\to\mathbb Q$ of $F$ 
(e.g. by embedding $F\subseteq F\otimes_k\overline{k}$ into the field of Puiseux series 
$\mathop{\underrightarrow{\lim}}\overline{k}((X^{1/N}))$). 
Let $S=\{m_1,\dots,m_d\}\subset\mathbb Q$ be the set of values of $v$ 
on $(\De)^{-1}\mathcal L\smallsetminus\{0\}$, and $m_1<\cdots<m_d$. Choose a basis $(Q_1\De,\dots,Q_d\De)$ 
of $\mathcal L$ such that $Q_i=X^{m_i}(1+\dots)$. Then $[Q_i\De,Q_j\De]=(m_j-m_i)X^{m_i+m_j-1}(1+\dots)\De$, 
so $m_i+m_j-1\in S$ for all $i<j$. As $m_1+m_2-1\ge m_1$, one has $m_2\ge 1$. Then $m_d>1$ (if $d\ge 3$), 
so $m_{d-1}+m_d-1>m_{d-1}$, and thus, $m_{d-1}+m_d-1=m_d$, which implies $m_{d-1}=1$. 
This means that $d=3$ (and $m_2=1$, so $m_1=1-m$ and $m_3=m+1$ for some $m>0$). 
Then $Q_2\De$ and $Q_3\De$ generate a two-dimensional Lie subalgebra $\mathcal L'\subset\mathcal L$. 

If $d=3$ then rescale $\eta_2\in\mathcal L'$ so that $[\eta_3,\eta_2]=\eta_3$. Then, as we have seen in 
the case $d=2$, $\eta_2=R\frac{\mathrm{d}}{\mathrm{d}R}$ and $\eta_3=\frac{\mathrm{d}}{\mathrm{d}R}$ 
for some $R\in F\smallsetminus k$. On the other hand, $[\eta_1,\eta_3]=a\eta_2+b\eta_3$ 
for some $a,b\in k$, i.e. $-\frac{\mathrm{d}}{\mathrm{d}R}(\eta_1/\frac{\mathrm{d}}{\mathrm{d}R}) 
\frac{\mathrm{d}}{\mathrm{d}R}=aR\frac{\mathrm{d}}{\mathrm{d}R}+b\frac{\mathrm{d}}{\mathrm{d}R}$, 
or $-\mathrm{d}(\eta_1/\frac{\mathrm{d}}{\mathrm{d}R})=\mathrm{d}(aR^2/2+bR)$, and thus, 
$\eta_1=-(aR^2/2+bR+c)\frac{\mathrm{d}}{\mathrm{d}R}$ for some $c\in k$, so $\mathcal L/\mathcal L'$ 
is spanned by $\frac{\mathrm{d}}{\mathrm{d}R^{-1}}$. \end{proof}

\subsection{Invariant subfields of \texorpdfstring{$F_{\Psi}$}{} for 
one-dimensional \texorpdfstring{$F|k$}{F|k}} \label{degtr=1} 
In this \S\ref{degtr=1}, we still assume that $F|k$ is a regular field extension of 
transcendence degree $1$. 
\begin{lemma} Let $\Psi$ be an infinite set, $F|k$ be a regular field extension of 
transcendence degree $1$, $K|k$ be a non-trivial $\Sy_{\Psi}$-invariant field extension 
in $F_{\Psi}$. Then the transcendence degree of $F_{\Psi}|K$ is finite. \end{lemma} 
\begin{proof} If $K$ contains an element of $F_S\smallsetminus k$ for a finite $S\subset\Psi$ 
then $\mathrm{tr.deg}(F_{\Psi}|K)\le\#S-1$. \end{proof}

\subsubsection{The case of characteristic \texorpdfstring{$0$}{}} 
Our next goal (Theorem~\ref{invar-subf} and Proposition~\ref{alg-sub-KPsi}) is to show that any invariant 
subfield of $F_{\Psi}|k$ is of type $(L_{\Psi})^H$ for some $L$ in $F|k$ and an algebraic $k$-group 
$H\subseteq\Aut_{\mathrm{field}}(L|k)$. 

First, we treat the case of $K$ algebraically closed in $F_{\Psi}$, i.e. with no non-trivial finite 
extensions in $F_{\Psi}$. The case of $K$ algebraically non-closed in $F_{\Psi}$ is 
considered in Proposition~\ref{alg-sub-KPsi} partially. 

Recall that, for a separable closure $\overline{k}$ of $k$, 
$(L_{\Psi})^H=((L_{k,\Psi}\otimes_k\overline{k})^{H(\overline{k})})^{\mathrm{Gal}(\overline{k}|k)} 
=(((L\otimes_k\overline{k})_{\overline{k},\Psi})^{H(\overline{k})})^{\mathrm{Gal}(\overline{k}|k)}$. 
\begin{theorem} \label{invar-subf} Let $\Psi$ be an infinite set, $F|k$ be 
a transcendence degree $1$ regular field extension of characteristic $0$, 
$K\supsetneqq k$ be an $\Sy_{\Psi}$-invariant subfield of $F_{\Psi}$ algebraically closed in $F_{\Psi}$. 

Then the transcendence degree $d$ of $F_{\Psi}$ over $K$ is $\le 3$. 
If $d=3$ then there exists a unique $R\in\mathrm{PGL}_2(k)\backslash(F\smallsetminus k)$ 
such that $K=k\left(\frac{(R(w)-R(x))(R(y)-R(z))}{(R(w)-R(z))(R(x)-R(y))}~|~w,x,y,z\in\Psi\right)
=(k(R)_{\Psi})^{\mathrm{PGL}_{2,k}}$. 
If $d=2$ then there exists a unique $R\in\mathbb P_k(F/k)$ such that 
$K=k\left(\frac{R(u)-R(w)}{R(u)-R(v)}~|~u,v,w\in\Psi\right)
=(k(R)_{\Psi})^{\mathbb G_{\mathrm{a},k}\rtimes\mathbb G_{\mathrm{m},k}}$. 

If $d=1$ then there is a directed 
system $(\pi_{ij}\colon W_i\to W_j)_{ij}$ of isogenies between torsors $W_i$ over geometrically irreducible 
one-dimensional algebraic $k$-groups $E_i$ endowed with a compatible system of $k$-field 
embeddings $\sigma_i:k(W_i)\hookrightarrow F\ ($i.e. $\sigma_i\pi_{ij}^*=\sigma_j$ for all $i,j)$ such that 
$K=\bigcup_i\left(k(W_i)_{\Psi}\right)^{E_i}\subset\bigcup_ik(W_i)_{\Psi}\subseteq F_{\Psi}$, where $E_i$ acts 
on $k(W_i)_{\Psi}$ diagonally.\footnote{The canonical difference morphism $W_i\times W_i\xrightarrow{-}E_i$ 
induces a natural isomorphism $k(E_i)_{\Psi}^{E_i}\xrightarrow{\sim}k(W_i)_{\Psi}^{E_i}$.} \end{theorem} 
\begin{proof} For each $x\in\Psi$, the subfield $F_x$ is preserved by the elements of 
$\mathrm{Der}_k(F_{\Psi})^{\Sy_{\Psi}|\{x\}}$, so the restrictions 
$\mathrm{Der}_{F_{\Psi\smallsetminus\{x\}}}(F_{\Psi})^{\Sy_{\Psi|\{x\}}}\xrightarrow{\rho_x}\mathrm{Der}_k(F_x)$ 
and $\mathrm{Der}_k(F_{\Psi})^{\Sy_{\Psi}}\xrightarrow{r_x}\mathrm{Der}_k(F_x)$ are well-defined. In fact, 
$\rho_x$ and $r_x$ are bijective. Denote by $\eta\mapsto\eta_x$ the inverse of the composition of $\rho_x$ 
with the isomorphism $\mathrm{Der}_k(F_x)\xrightarrow[\sim]{\iota_x}\mathrm{Der}_k(F)$ (induced by 
$(x)\colon F\xrightarrow{\sim}F_x$). The composition $\iota_x\circ r_x$ is independent of $x$, 
and thus, gives a natural Lie $k$-algebra isomorphism 
$\mathrm{Der}_k(F)\xrightarrow{\sim}\mathrm{Der}_k(F_{\Psi})^{\Sy_{\Psi}}$, $\eta\mapsto\sum_{u\in\Psi}\eta_u$.

It follows from the preceeding discussion, that $K$ is the common kernel of 
the elements of a finite-dimensional Lie $k$-subalgebra $\mathcal L$ of $\mathrm{Der}_k(F)$: 
$K=\bigcap_{\eta\in\mathcal L}\ker(\sum_{x\in\Psi}\eta_x)$.

If $d=2$ then, by Proposition~\ref{restr-fin-dim}, $K$ is contained in the common kernel of 
$\sum_{x\in\Psi}\frac{\partial}{\partial R(x)}$ and $\sum_{x\in\Psi}R(x)\frac{\partial}{\partial R(x)}$ 
for some $R\in F\smallsetminus k$ with a well-defined class of $R$ in $\mathbb P_k(F/k)$. This means 
that $K$ contains the intersection of the subfield generated over $k$ by $R(x)-R(y)$ for all 
$x,y\in\Psi$ and the subfield generated over $k$ by $R(x)/R(y)$ for all $x,y\in\Psi$. 

If $d=3$ then, by Proposition~\ref{restr-fin-dim}, $K$ is the common kernel of 
$\sum_{x\in\Psi}\frac{\partial}{\partial R(x)}$, $\sum_{x\in\Psi}R(x)\frac{\partial}{\partial R(x)}$ 
and $\sum_{x\in\Psi}R(x)^2\frac{\partial}{\partial R(x)}$ for some $R\in F\smallsetminus k$. This means 
that $K$ is the intersection of the subfields $k\left(R(x)-R(y)~|~x,y\in\Psi\right)$, 
$k\left(R(x)/R(y)~|~x,y\in\Psi\right)$, and $k\left(1/R(x)-1/R(y)~|~x,y\in\Psi\right)$. 

If $d=1$, we invoke the Lefschetz principle, i.e. reduce the problem to the case 
of a countable $k$ and embed $k$ into $\mathbb C$. Let $0\neq\eta\in\mathrm{Der}_k(F)$ be such that 
$\sum_{x\in\Psi}\eta_x$ annihilates $K$. For each $f\in K$, there is a smooth projective curve $\overline{C}$ 
over $k$ with an embedding $k(\overline{C})\hookrightarrow F$ and a finite subset $S\subset\Psi$ such that 
$f\in k(\overline{C})_S\subset k(\overline{C})_{\Psi}$ and $k(\overline{C})$ is $\eta$-invariant (e.g., if 
$\eta=\varphi\frac{\mathrm{d}}{\mathrm{d}X}$ then any $k(\overline{C})$ containing $\varphi$ and $X$ is so). 
Let $C\subseteq\overline{C}$ be the complement to the set of poles of the 1-form $1/\eta$. If considered 
locally in the analytic topology on $C$, $1/\eta=\mathrm{d}u$ for an analytic function $u(x)=\int_*^x1/\eta$. 
Fix some $v\in\Psi$. Then $f$ is an analytic function in variables $T_i$, enumerated by 
$i\in S\smallsetminus\{v\}$, that is evaluated at $T_i=u_i-u_v$. As the path from $\ast$ to $x$ varies, the 
values of the branches of $u(x)$ are modified by the periods of the 1-form $1/\eta$. As $f$ is non-constant, 
there are 3 options for the periods: (1) they are zero; (2) they form a cyclic group $\mathbb Z\cdot\omega$ 
for some $\omega\in\mathbb C^{\times}$; (3) they form a lattice $L$ in $\mathbb C$. 
In the case (1) $u(x)$ is rational and $f\in\mathbb C(u(x)-u(y)~|~x,y\in\Psi)$. In the case (2) 
$Q(x):=\exp(\frac{2\pi i}{\omega}u(x))$ is rational and $f\in\mathbb C(Q(x)/Q(y)~|~x,y\in\Psi)$. 
In the case (3), let $\wp_L(z):=\frac{1}{z^2}+\sum_{\lambda\in L\smallsetminus\{0\}}
\left(\frac{1}{(z-\lambda)^2}-\frac{1}{\lambda^2}\right)$ be the Weierstra\ss\ $\wp$-function, so 
$\wp'_L(z)^2=4\wp_L(z)^3+a\wp_L(z)+b$ for some $a,b\in\mathbb C$. Then $g(x):=\wp_L(u(x))$, 
$h(x):=\wp'_L(u(x))$, $\wp_L(u(x)-u(y))=\left[\frac{1}{2}\frac{h(x)+h(y)}{g(x)-g(y)}\right]^2-g(x)-g(y)$ 
and $\wp'_L(u(x)-u(y))=
\frac{1}{2}(h(x)+h(y))\left[\frac{(g(x)-g(y))(6g(x)^2+a)-(h(x)+h(y))h(x)}{(g(x)-g(y))^3}\right]-h(x)$ 
are rational and $f\in\mathbb C(\wp_L(u(x)-u(y)),\wp'_L(u(x)-u(y))~|~x,y\in\Psi)$. Thus, all such $C$'s 
are one-dimensional algebraic groups, and $1/\eta$ is an invariant 1-form on each of them, i.e. the 
inclusions of the function fields of the curves $C$ are induced by isogenies between them. Coming back 
to the original field $k$, which may not be algebraically closed, the curves $C$ become torsors under 
algebraic groups. \end{proof}

\subsubsection{Invariant subfields of \texorpdfstring{$F_{\Psi}$}{} that are not algebraically closed} 
It is clear that the non-connected $k$-groups of automorphisms of $F|k$ fix in $F_{\Psi}$ invariant 
subfields that are not relatively algebraically closed. Appendix \ref{group-actions-on-fields} provides 
examples of the same type of subfields of $F_{\Psi}$ that are fixed by a {\sl connected 
but non-reduced} $k$-group. 
\begin{lemma} \label{symm-subgrps} Let $\Psi$ be an infinite set, $H$ be a Hausdorff group, 
and $\Gamma$ be a closed subgroup of $\Maps(\Psi,H)$ normalized by $\Sy_{\Psi}$. Then 
$\Gamma=H_1\Maps(\Psi,H_0)$ for a closed subgroup $H_0\subseteq H$ and a closed subgroup $H_1$ 
normalizing $H_0$, where $H_1$ is embedded into $\Maps(\Psi,H)$ as constant maps. \end{lemma} 
\begin{proof} For any $\gamma\in\Gamma$, $\Gamma$ contains 
$\gamma^{-1}\gamma^{\iota}$, where $\iota\in\Sy_{\Psi}$ is a transposition, say of $x$ and $y$. For any 
finite subset $S\subset\Psi$ there is a transposition of $y$ and an element of $\Psi\smallsetminus S$. 
Therefore, $\Gamma$ contains the element $\gamma'$ with $\gamma'(x)=\gamma(x)^{-1}\gamma(y)$ and $\gamma'(w)=1$ 
for all $w\in\Psi\smallsetminus\{x\}$. Let $H_0$ (respectively, $H_1$) be the closed subgroup of $H$ 
generated by the elements $\gamma(x)^{-1}\gamma(y)$ (respectively, by the elements $\gamma(x)$) for all 
$\gamma\in\Gamma$. Then obviously $H_1$ normalizes $H_0$ and $\Maps(\Psi,H_0)\subseteq\Gamma$. \end{proof} 

\begin{proposition} \label{sub-algebraic} Let $\Psi$ be an infinite set, $F|k$ be a regular field 
extension, $K$ be an $\Sy_{\Psi}$-invariant field extension of $k$ in $F_{\Psi}$ such that $F_{\Psi}|K$ is 
algebraic. Then $K$ contains $F'_{\Psi}$ for an intermediate subfield $F'$ in $F|k$ with algebraic $F|F'$. 

If the characteristic of $k$ is $0$ then $K=\left(L_{\Psi}\right)^{\Gamma}$, 
where $L$ is an intermediate subfield in $F|k$ and $\Gamma$ is a profinite algebraic 
$k$-group of automorphisms of $L|k$ acting diagonally on $L_{\Psi}$. More explicitly, 
$K=\left((L\otimes_k\overline{k})_{\Psi}^{\Gamma(\overline{k})}\right)^{\mathrm{Gal}(\overline{k}|k)}$, 
where $\overline{k}$ is an algebraic closure of $k$. \end{proposition} 
\begin{proof} Fix a transcendence basis $B$ of $F|k$, and some $x\in\Psi$. For each $b\in B$, 
consider the minimal polynomial $X^d+\sum_{i=0}^{d-1}a_iX^i\in K[X]$ over $K$ of the element 
$b(x)$ of $F_x\smallsetminus k$, where $F_x$ and $(x)$ are defined on p.\pageref{tensor-factors}. 
Clearly, $a_i\in F_{\Psi}^{\Sy_{\Psi|\{x\}}}\cap K=F_x\cap K$ for all $i$. As $b(x)$ is 
transcendental over $k(B\smallsetminus\{b\})$, so is $a_s$ for some $s$, and therefore, 
sending $b(x)$ to an appropriate coefficient of its minimal polynomial over $K$ induces 
an $\Sy_{\Psi}$-field embedding $k(B)_{\Psi}\hookrightarrow K$ over $k$. 

If $B$ is separating (i.e. $F$ is separable over $k(B)$), let $\widetilde{F}$ be a Galois 
normalization of $F$ over $k(B)$ with the Galois group denoted $H$ (finite if $F|k$ is finitely generated), 
and $k'$ be the algebraic closure of $k$ in $\widetilde{F}$, so $\widetilde{F}_{\Psi}=\widetilde{F}_{k',\Psi}$. 
Denote by $\Maps^{\circ}(\Psi,H)$ the group of those maps that are constant when composed with the restriction 
$H\to\mathrm{Gal}(k'|k)$. Then $\mathrm{Gal}(\widetilde{F}_{\Psi}|k(B)_{\Psi})=\Maps^{\circ}(\Psi,H)$, while 
$\mathrm{Gal}(\widetilde{F}_{\Psi}|K)$ is a closed subgroup normalized by $\Sy_{\Psi}$ 
and containing $\Maps(\Psi,\mathrm{Gal}(\widetilde{F}|F))$.

By Lemma \ref{symm-subgrps}, $\mathrm{Gal}(\widetilde{F}_{\Psi}|Kk')=H_1\Maps(\Psi,H_0)$ 
for a closed subgroup $H_0\subseteq H$ and a closed subgroup $H_1\subseteq N_H(H_0)$, where 
$H_1$ is embedded into $\Maps(\Psi,H)$ as constant maps. As $K\subseteq F_{\Psi}$, one 
has $\mathrm{Gal}(\widetilde{F}_{\Psi}|K)\supseteq\mathrm{Gal}(\widetilde{F}_{\Psi}|F_{\Psi})$, 
and therefore, $\mathrm{Gal}(\widetilde{F}_{\Psi}|Kk')\supseteq\mathrm{Gal}(\widetilde{F}_{\Psi}|F_{\Psi}k')$, 
or equivalently, $H_1\Maps(\Psi,H_0)\supseteq\Maps(\Psi,\mathrm{Gal}(\widetilde{F}|F))$. 

In particular, $H_0\supseteq\mathrm{Gal}(\widetilde{F}|F)$, and therefore, $L:=\widetilde{F}^{H_0}$ 
is a subfield of $F$. Then $H_1$ preserves $L$ and acts on it via its quotient 
$\Gamma:=H_1/(H_0\cap H_1)$, which is a closed subgroup of the compact group $\Aut_{\mathrm{field}}(L|k(B))$, 
while $K$ is the fixed field of $\Gamma$ under the diagonal action on $L_{\Psi}$. \end{proof} 

\begin{lemma} \label{rat-mult-add} Let $k$ be a field of characteristic $0$, $E$ be a geometrically 
connected one-dimensional algebraic $k$-group, and $E\xrightarrow{Q}X$ be a morphism to 
a $k$-curve. Suppose that the composition of $E\times E\xrightarrow{Q\times Q}X\times X$ 
with a rational map $X\times X\stackrel{P}{\dashrightarrow}Y$ to a $k$-curve is 
non-constant and factors through $E\times E\xrightarrow{-}E$. 

Then $Q$ is the composition of the projection $E\to E/\Gamma$ for some finite 
algebraic $k$-subgroup $\Gamma\subset E$ and an embedding $E/\Gamma\hookrightarrow X$. 

In particular, {\rm (i)} $Q$ is an embedding if $E=\mathrm{Spec}(k[T])$ is the 
additive group, i.e., $Q^*k(X)=k(T);$ {\rm (ii)} $Q^*k(X)=k(T^n)$ for some integer 
$n\ge 1$ if $E=\mathrm{Spec}(k[T,T^{-1}])$ is the multiplicative group. \end{lemma} 
\begin{proof} We may assume that $Y$ is smooth and projective. Let $a,b\in E(K)$ be points over 
a field extension $K|k$ such that $Q(a)=Q(b)$ and $X\times\{Q(a)\}$ contains no indeterminacy 
points of $P$. Then, for a morphism $S\colon E\to Y$, 
$S(u-a)=P(Q(u),Q(a))=P(Q(u),Q(b))=S(u-b)\in Y(K(u))$. As $S$ is not constant, 
$a-b$ is a torsion element. Taking $\Gamma:=\{a-b~|~a,b\in E(K),\ Q(a)=Q(b)\}$, we get the assertion. \end{proof} 

For a regular extension $F|k$ of transcendence degree 1 and characteristic 0, Theorem~\ref{invar-subf} 
describes algebraic closures $L$ in $F_{\Psi}$ of $\Sy_{\Psi}$-invariant extensions $K$ of $k$ in 
$F_{\Psi}$. Now the task is to describe all possible $K$'s for each $L$. In the setting of 
Theorem~\ref{invar-subf}, we consider only the case of trivial torsors. Lemma~\ref{triv_torsors} and 
Example~\ref{exa-triv_torsors} provide situations where all torsors are trivial anyway. 
\begin{proposition} \label{alg-sub-KPsi} Let $\Psi$ be an infinite set, $k$ be a field of characteristic 
$0$, and $L$ be one of the fields $\Ka:=(k(E)_{\Psi})^E$ for a geometrically irreducible one-dimensional 
$k$-group $E$, $\Ka':=\bigcup_i(k(E_i)_{\Psi})^{E_i}$ for a compatible directed system of isogenies of 
one-dimensional $k$-groups $E_i$, 
$\Kc:=k\left(\frac{u-v}{v-w}~|~\mbox{ {\rm pairwise distinct} }u,v,w\in\Psi\right)$, 
$\Kd:=k\left(\frac{(t-u)(v-w)}{(v-u)(t-w)}~|~\mbox{ {\rm pairwise distinct} }t,u,v,w\in\Psi\right)$. 

Let $K$ be an $\Sy_{\Psi}$-invariant field extension of $k$ in $L$ over which $L$ is algebraic. 

Then {\rm (a)} $K=L^{\Gamma^{\Psi}\rtimes H}$ for a finite $k$-subgroup $\Gamma\subset E$ 
normalized by a finite $k$-group $H$ of $k$-group automorphisms of $E$ if $L=\Ka;$\footnote{e.g., 
$K=\Kc((u-v)^n~|~u,v\in\Psi)$ for an integer $n\ge 1$ if $E=\mathbb G_{\mathrm{a},k};$ 
$H$ is a subgroup of the cyclic group $\mathrm{End}_{k\mbox{{\rm -group}}}(E)^{\times}$ 
of order dividing and less than 12, if $E$ is a torus or an elliptic curve.} 
{\rm (a}$')\ K=\bigcup_i(k(\Lambda\backslash E_i)_{\Psi})^{E_i\rtimes H}$ for a profinite subgroup 
$\Lambda\subseteq\prlim_iE_i(\overline{k})_{\mathrm{tor}}$ normalized by 
a finite $k$-group $H\subseteq(\mathrm{End}_{k\mbox{{\rm -group}}}(E_i)\otimes\mathbb Q)^{\times}$ of $k$-group 
automorphisms of $\prlim_iE_i$, if $L=\Ka';$ {\rm (c,d)} $K=L$ if $L=\Kc$ or $L=\Kd$. \end{proposition} 
\begin{proof} Fix some pairwise distinct $x,y,z\in\Psi$. Set $\Psi':=\Psi\smallsetminus\{x\}$ if $L=\Ka$ 
or $\Ka'$; $\Psi':=\Psi\smallsetminus\{x,y\}$ if $L=\Kc$; $\Psi':=\Psi\smallsetminus\{x,y,z\}$ if $L=\Kd$. 

If considered merely as an $\Sy_{\Psi'}$-field, $L$ is identified with $\widetilde{F}_{\Psi'}$ with 
the natural $\Sy_{\Psi'}$-action, where $\widetilde{F}=k(E)$ if $L=\Ka$, $\widetilde{F}=\bigcup_ik(E_i)$ 
if $L=\Ka'$, and $\widetilde{F}=k(\xi)$ (so $L=k(\xi_u~|~u\in\Psi')$) if $L$ is either $\Kc$ 
(with $\xi_u:=\frac{u-y}{x-y}$) or $\Kd$ (with $\xi_u:=\frac{(u-y)(z-x)}{(z-u)(x-y)}$), 
since $\Ka=k(E\backslash E^{\Psi})=k(E^{\Psi'})$, $\Ka':=\bigcup_ik(E_i^{\Psi'})$, 
$\Kc=k\left(\begin{pmatrix}\mathbb G_{\mathrm{m},k}&\mathbb G_{\mathrm{a},k}\\
0&1\end{pmatrix}\backslash(\mathbb A^1_k)^{\Psi}\right)=k\left((\mathbb A^1_k)^{\Psi'}\right)$, 
$\Kd=k\left(\mathrm{PGL}_{2,k}\backslash(\mathbb P^1_k)^{\Psi}\right)=k\left((\mathbb P^1_k)^{\Psi'}\right)$, 
where all actions on $(-)^{\Psi}$ are diagonal. 

Suppose that $H$ is a finite group acting on a smooth $\Sy_{\Psi}$-set $S$ so that the $H$-action is normalized 
by $\Sy_{\Psi}$. The centralizer of $H$ in $\Sy_S$ (i.e. $\bigcap_{h\in H}\{g\in\Sy_S~|~ghg^{-1}=h\}$) is closed. 
For any open proper subgroup of $\Sy_{\Psi}$, the finite intersections of its conjugates form a base of open 
subgroups. Thus, $H$ commutes with an open subgroup of $\Sy_{\Psi}$, i.e. the conjugation homomorphism 
$\Sy_{\Psi}\to\Aut(H)$ is continuous and non-injective. As $\Sy_{\Psi}$ is topologically simple, the $H$-action 
actually commutes with the whole $\Sy_{\Psi}$. In particular, if $S$ looks as $\widetilde{F}_{\Psi'}$ 
then $H$ can be identified with a finite subgroup of $\Aut_{\mathrm{field}}(\widetilde{F}|k)$ acting diagonally.

By Proposition~\ref{sub-algebraic} (with $\Psi$ replaced by $\Psi'$) 
there is an intermediate subfield $F$ in $\widetilde{F}|k$ and an algebraic profinite 
$k$-subgroup $H\subseteq\Aut_{\mathrm{field}}(F|k)$ such that 
$K=\left(F_{\Psi'}\right)^H$. Set $F':=F^H$. Clearly, $F'_{\Psi'}\subseteq K\subseteq F_{\Psi'}$. 
In the case $L=\Ka$, let $X$ and $Y$ be smooth projective curves over $k$ with, respectively, $k(X)=F$ 
and $k(Y)=F'$. Let $E\xrightarrow{Q}X\xrightarrow{R}Y$ be the morphisms induced by the inclusions 
$F'\subseteq F\subseteq\widetilde{F}$. The transposition $(xu)\in\Sy_{\Psi}$ on $\widetilde{F}_{\{u,v\}}$ 
is induced by the involution $E\times E\xrightarrow{(A_u,B_v)\mapsto(-A_u,B_v-A_u)}E\times E$. 
Then $(xu)^*F'_v\subseteq(xu)^*F'_{\{u,v\}}\subseteq(xu)^*F_{\{u,v\}}\cap K\subseteq(xu)^*F_{\{u,v\}}
\subseteq\widetilde{F}_{\{u,v\}}$, so $(xu)^*F'_v\subseteq\widetilde{F}_{\{u,v\}}\cap K\subseteq F_{\{u,v\}}$. 
This corresponds to fact that the composition $E\times E\xrightarrow{-}E\xrightarrow{Q}X\xrightarrow{R}Y$ 
factors through $E\times E\xrightarrow{Q\times Q}X\times X$. 
This also determines the $\Sy_{\Psi}$-action on $\widetilde{F}_{\Psi'}$ in the case $L=\Ka'$. 
In the cases (c,d), the transposition $(xu)\in\Sy_{\Psi}$ transforms $\xi_v$ to $\xi_v/\xi_u$, 
so these cases fit into the above scheme with $E$ being the multiplicative group. 

By Lemma~\ref{rat-mult-add}, in the cases (a,a$'$,c,d) $k(X)=k(E)^{\Gamma}$ for some finite 
$k$-subgroup $\Gamma\subset E$, and therefore, the group $\Sy_{\Psi}$ acts on $k(X)_{\Psi'}$, 
and $K=(k(X)_{\Psi'})^H$ for a finite group $H$ of $\Sy_{\Psi}$-automorphisms of $k(X)_{\Psi'}$. 
In particular, $X=E$ if $E$ is the additive group, $X=\mathrm{Spec}(k[T^{\pm n}])$ for 
some integer $n\ge 1$ if $E=\mathrm{Spec}(k[T^{\pm 1}])$ is the multiplicative group. 
There are more options if $E$ is an elliptic curve. 

Any $\Sy_{\Psi'}$-field automorphism of $k(E)_{\Psi'}|k$ comes from an automorphism of $k(E)|k$. 
It is easy to see that any $\Sy_{\Psi}$-field automorphism of $\Ka\cong k(E)_{\Psi'}$ over $k$ comes 
from an automorphism of the $k$-group $E$. Then $H$ is the group $\mu_n$ of $n$-th roots of unity for 
some integer $n\ge 1$ acting on $k(\Psi)$ by $\zeta\colon u\mapsto\zeta u$ if $E$ is the additive group; 
if $H$ is non-trivial then its only non-trivial automorphism acts on $k(\xi_u~|~u\in\Psi')$ by 
$\xi_u\mapsto\xi_u^{-1}$ if $E$ is the multiplicative group; 
$H$ is a cyclic group of order 1, 2, 3, 4, or 6 if $E$ is an elliptic curve. 

In the cases (a,c,d) $K$ contains $R(Q(\xi_u/\xi_v))$. This means that $R(Q(\xi_u/\xi_v))=P(Q(\xi_u),Q(\xi_v))$ 
for a rational function $P$ over $k$ in two variables. By Lemma~\ref{rat-mult-add}, $k(Q(T))=k(T^n)$ 
for some integer $n\ge 1$, so $k(Q(\xi_u)~|~u\in\Psi')=k(\xi_u^n~|~u\in\Psi')$. 

The case of $\Ka'$ follows, since $\Ka'$ is a union of subfields of type $\Ka$, while the only 
non-trivial automorphism of a subgroup of $\mathbb Q$ corresponds to the involution $u\mapsto u^{-1}$. 

In the cases (c,d), we already know that $k(Q(T))=k(T^n)$ for some integer $n\ge 1$. 
As the transposition $(xy)\in\Sy_{\Psi}$ transforms $\xi_u$ to $1-\xi_u$, $K$ 
contains $R(Q(1-\xi_u))$. Then $f(u):=R(Q(1-\xi_u))=P(Q(\xi_u))$ for some $P\in k(T)$, 
so $f(1-\zeta+u)=f(\zeta^{-1}u)=f(u)$ for any $\zeta\in\mu_n$, so $f\in k$ unless $n=1$. 
Now, in the case (c), if an automorphism $\xi_u\mapsto\frac{a\xi_u+b}{c\xi_u+d}$ of $\Kc|k$ commutes 
with the transpositions $(xy)$ and $(ux)$ then $\frac{(c-a)T+d-b}{cT+d}=\frac{aT-a-b}{cT-c-d}$ 
and $\frac{cT+d}{aT+b}=\frac{a+bT}{c+dT}$ (as $\xi_u^{(ux)}=\xi_u^{-1}$). 
The only non-identical element of $\mathrm{PGL}_{2,k}$ satisfying these conditions is $\frac{T-2}{2T-1}$. 
The corresponding automorphism of $\Kc$ would be the involution $\frac{u-v}{v-w}\mapsto\frac{u+v-2w}{v+w-2u}$. 
However, this involution is not multiplicative: $-\frac{t-v}{v-w}=\frac{u-v}{v-w}\cdot\frac{t-v}{v-u}$, 
but $-\frac{t+v-2w}{v+w-2t}\neq\frac{u+v-2w}{v+w-2u}\cdot\frac{t+v-2u}{v+u-2t}$. 

In the case (d), we already know that $k(Q(T))=k(T)$. 
It remains to show that any $\Sy_{\Psi}$-automorphism $\alpha$ of $\Kd|k$ is trivial. 
For all $u\in\Psi'$, $\alpha$ acts on $\xi_u$'s by a common element of $\mathrm{PGL}_2(k)$. As $\alpha$ 
commutes with $\Sy_{\Psi}$, it acts in the same way on the whole $\Sy_{\Psi}$-orbit of $\xi_u$. 
As $\frac{(x-u)(y-z)}{(x-y)(u-z)}\frac{(u-z)(x-t)}{(u-x)(z-t)}=\frac{(x-t)(y-z)}{(x-y)(t-z)}$, 
$\alpha$ should be multiplicative, i.e. $\alpha$ is given by $T^{\pm 1}\in\mathrm{PGL}_2(k)$. 
But $\xi_u-\xi_v=\frac{(v-u)(y-z)(x-z)}{(x-y)(u-z)(v-z)}$, so 
$\frac{\xi_u-\xi_v}{\xi_u-\xi_w}=\frac{(w-z)(v-u)}{(w-u)(v-z)}$, while 
$\xi_u^{-1}-\xi_v^{-1}=\frac{(x-y)(x-z)(u-v)}{(x-u)(x-v)(y-z)}$, so 
$\frac{\xi_u^{-1}-\xi_v^{-1}}{\xi_u^{-1}-\xi_w^{-1}}=\frac{(x-w)(u-v)}{(x-v)(u-w)}
\neq\frac{(w-u)(v-z)}{(w-z)(v-u)}$. Therefore, $\alpha$ should be trivial. \end{proof}

\section{Symmetric irreducible subvarieties of infinite cartesian powers of curves} 
\label{subvar-of-curve-powers} 
Let $X$ be a geometrically irreducible curve over a field $k$. The following result, which is 
essentially \cite[Proposition 2.7]{NS}, shows that the symmetric integral $k$-subschemes of 
infinite cartesian powers of $X$, other than themselves, are contained in $X$ embedded diagonally. 
\begin{proposition} Let $F|k$ be a regular field extension of transcendence degree $1$, 
and $\Psi$ be an infinite set. Then the only non-zero $\Sy_{\Psi}$-invariant prime ideal 
in $\mathcal O_{\Psi}:=\bigotimes_{k,\ u\in\Psi}F$ is the kernel of the multiplication map 
$\mathcal O_{\Psi}\to F$, $\otimes_{u\in\Psi}f_u\mapsto\prod_{u\in\Psi}f_u$. \end{proposition} 
\begin{proof} Let an $\Sy_{\Psi}$-invariant prime ideal $\mathfrak{p}\subset\mathcal O_{\Psi}$ 
contain a non-zero element in $\bigotimes_{k,\ u\in S}F\subset\mathcal O_{\Psi}$ for a finite subset 
$S\subset\Psi$. Then the fraction field $K$ of $\mathcal O_{\Psi}/\mathfrak{p}$ is of transcendence 
degree $<\#S$ over $k$, so the automorphism group of $K$ over $k$ is locally compact. As (i) $\Sy_{\Psi}$ 
is not locally compact, (ii) the finite intersections of the conjugates of an arbitrary open proper subgroup 
of $\Sy_{\Psi}$ form a base of its open subgroups, the $\Sy_{\Psi}$-action on $\mathcal O_{\Psi}/\mathfrak{p}$ 
is not faithful. Then, as $\Sy_{\Psi}$ is topologically simple, the $\Sy_{\Psi}$-action on 
$\mathcal O_{\Psi}/\mathfrak{p}$ is trivial, i.e. $f(u)\equiv f(v)\pmod{\mathfrak{p}}$ for all $f\in F$, 
and $u,v\in\Psi$. \end{proof} 

\begin{remark} Let $F|k$ be a regular field extension, $\Psi$ be an infinite set, and $\mathfrak{p}$ be 
an $\Sy_{\Psi}$-invariant prime ideal in $\mathcal O_{\Psi}:=\bigotimes_{k,\ u\in\Psi}F$. It is plausible 
that $\mathfrak{p}$ is the kernel of the natural map $\mathcal O_{\Psi}\to\bigotimes_{L,\ u\in\Psi}F$ 
for an intermediate subfield $L$ in $F|k$ which is algebraically closed in $F$, so that 
$A:=\mathcal O_{\Psi}/\mathfrak{p}$ is isomorphic to $\bigotimes_{L,\ u\in\Psi}F$. 

In other words, if $X$ is a geometrically irreducible variety over $k$ then any integral 
$\Sy_{\Psi}$-invariant subscheme of $X^{\Psi}$ is the $\Psi$-th fibre power $W^{\Psi}_Y$ of a subvariety 
$W$ of $X$ over a dominant rational map $W\dasharrow Y$ with geometrically irreducible general fibre. 

Let us show that $A=\bigotimes_{B,\ u\in\Psi}(BF)$, where $B:=A^{\Sy_{\Psi}}$ is 
algebraically closed in $A$, and $BF$ is an integral quotient ring of $B\otimes_kF$. 

\begin{proof} The ring $B$ is algebraically closed in $A$, since the $\Sy_{\Psi}$-orbit of any element in 
$A$ algebraic over $B$ is finite, while there are no proper open subgroups of finite index in $\Sy_{\Psi}$. 

Fix some $u\in\Psi$. Let us show by induction on $n\ge 0$ (the case $n=0$ being trivial) that 
any $h_1,\dots,h_n\in F_u\subset A$ algebraically independent over $B$ are also algebraically 
independent over the $B$-subalgebra $\widetilde{C}$ of $A$ generated by the image in $A$ of 
$\bigotimes_{k,\ \Psi\smallsetminus\{u\}}F$. Indeed, otherwise $h_1,\dots,h_n$ are algebraically dependent 
over the $B$-subalgebra $C$ of $A$ generated by the image in $A$ of $\bigotimes_{k,\ u\in S}F$ 
for a finite set $S\subset\Psi\smallsetminus\{u\}$. This dependence is given by an irreducible polynomial 
$P$ over the fraction field of $C$. By induction hypothesis, this $P$ is unique, if we fix one of its 
monomials and require this monomial to have coefficient 1. However, the coefficients of $P$ belong to 
the fraction field of $C$, so they are fixed by $\Sy_{\Psi|S}$. Fix some $g\in\Sy_{\Psi|\{u\}}$ such that 
$S\cap g(S)=\varnothing$. As $P=P^g$, while $\Sy_{\Psi|S}$ and $\Sy_{\Psi|g(S)}$ generate $\Sy_{\Psi}$, 
the coefficients of $P$ are fixed by $\Sy_{\Psi}$, i.e. they belong to $B$, contradicting our assumption. 
\end{proof} \end{remark}

\section{Some finitely generated \texorpdfstring{$\Sy_{\Psi}$}{}-extensions of fields 
of type \texorpdfstring{$F_{\Psi}$}{}} \label{Some-fin-gen-field-exts} 
Any automorphism group $G$ of any field $K$ acts naturally on the set of birational types of 
$K$-varieties. If the $G$-action on $K$ extends to the function field $K(\mathbb W)$ of a $K$-variety 
$\mathbb W$ then the birational type of $\mathbb W$ is fixed by the natural $G$-action. An obvious 
source of such extended $G$-actions is given by varieties $\mathbb W=W\times_{K^G}K$ (called, along 
with their function fields, {\sl isotrivial}) for all geometrically irreducible $K^G$-varieties $W$. 

Thus, to certain extent, the study of extensions of the $G$-action on $K$ to $G$-actions 
on the function fields over $K$ splits to (i) the study of the birational types fixed by 
the natural $G$-action, (ii) the study of $G$-actions on isotrivial extensions of $K$. 

\subsection{Isotriviality of function fields of varieties of general type and of curves} \label{isotriv} 
For any {\sl permutation group} $G$ and a smooth $G$-field $K$, denote by $\mathrm{Pic}_K(G)$ the group (under 
tensor product over $K$) of isomorphism classes of objects of $\Sm_K(G)$ that are one-dimensional over $K$. 

For each smooth $G$-module $M$, set 
$H^*(G,M):=\mathrm{Ext}^*_{\Sm_{\mathbb Z}(G)}(\mathbb Z,M)$. 
Obviously, $\mathrm{Pic}_K(G)$ is canonically isomorphic to $H^1(G,K^{\times})$. 

We are mainly interested in the case of $G=\Sy_{\Psi}$. By \cite[Theorem 3.5 and Corollary 3.7]{EvansHewitt}, 
if a finite abelian group $A$ is considered 
as a trivial $\mathfrak{S}_{\Psi}$-module then $H^{>0}(\Sy_{\Psi},A)=0$. 

\begin{lemma} \label{Pic-torsion} Let $G$ be a permutation group, $K$ be a smooth $G$-field and $n>0$ 
be an integer. Set $k:=K^G$ and $\mu_n:=\{z\in K^{\times}~|~z^n=1\}$. Then there is a natural exact sequence 
\[H^1(G,\mu_n)\to{}_n\mathrm{Pic}_K(G)\xrightarrow{\beta}
(K^{\times}/K^{\times n})^G/k^{\times}\xrightarrow{\xi}H^2(G,\mu_n).\] 
\end{lemma} 
\begin{proof} In terms of 1-cocycles $G\xrightarrow{\sigma\mapsto f_{\sigma}}K^{\times}$, define $\beta$ 
by $\beta((f_{\sigma}))=[a]$, where 
$a^{\sigma}/a=f_{\sigma}^n$ for all $\sigma\in G$.\footnote{If $\mathcal L$ is an object of 
$\Sm_K(G)$ with $\mathcal L^{\otimes_K^n}\cong K$, choose non-zero elements $\pi\in\mathcal L$ and 
$\lambda\in(\mathcal L^{\otimes_K^n})^G$, and set $\beta([\mathcal L]):=[\pi^{\otimes n}/\lambda]$.} 

Define $\xi$ by $[a]\mapsto[(b_{\sigma}b_{\tau}^{\sigma}b_{\sigma\tau}^{-1})]$, where 
$a^{\sigma}/a=b_{\sigma}^n$ for some 1-cochain $(b_{\sigma})$ with values in $K^{\times}$. 
If $\xi a=0$ then there exists a 1-cochain $(\zeta_{\sigma})$ with values in $\mu_n$ such that 
$(b_{\sigma}\zeta_{\sigma})$ is a 1-cocycle with values in $K^{\times}$, so it defines an element of 
$\mathrm{Pic}_K(G)$. Obviously, it is of order $n$, while $\beta$ maps it to $a$. Clearly, $\xi\beta=0$. 

If $\beta((f_{\sigma}))=0$ then there exists $b\in K^{\times}$ such that $(b^n)^{\sigma}/b^n=f_{\sigma}^n$ 
for all $\sigma\in G$, and therefore, $(bf_{\sigma}/b^{\sigma})$ is a 1-cocycle with values in $\mu_n$. 
This shows that our sequence is exact. \end{proof} 

\begin{proposition} \label{isotriv_gen-type_curves} Let $G$ be a permutation group, and $L|K$ be a smooth 
$G$-field extension. Suppose that \begin{enumerate} \item \label{generat-fin-dim} 
any smooth $K\langle G\rangle$-module finite-dimensional over $K$ is a direct sum of copies of $K;$ 
\item $L$ is the function field {\rm either} of a smooth projective genus $0$ curve $\mathbb W$ over $K$, 
{\rm or} of a normal projective variety $\mathbb W$ over $K$ such that the pluricanonical map 
$\mathbb W\dashrightarrow\mathbb P(\Gamma(\mathbb W,\omega^{\otimes s}_{\mathbb W|K})^{\vee})$ 
is generically injective for some $s>0$. \end{enumerate} 
Then $L^G=k(W)$ for an irreducible variety $W$ over $k:=K^G$ and $L=K(W)$. \end{proposition} 
\begin{proof} For each integer $s$, let $\Gamma_s:=\Gamma(\mathbb W,\omega^{\otimes s}_{\mathbb W|K})$ 
be the (finite-dimensional) $K$-vector space of global sections of the pluricanonical sheaf 
$\omega^{\otimes s}_{\mathbb W|K}$ on $\mathbb W$. The group $G$ (i) naturally acts $L$-semilinearly 
and smoothly on the one-dimensional $L$-vector space $(\Omega^{\dim\mathbb W}_{L|K})^{\otimes^s_L}$, 
(ii) preserves the $K$-vector space $\Gamma_s\subset(\Omega^{\dim\mathbb W}_{L|K})^{\otimes^s_L}$, 
and acts on it $K$-semilinearly. 

By the condition (\ref{generat-fin-dim}), the $k$-vector space $\Gamma_s^G$ is a $k$-lattice in 
the $K$-vector space $\Gamma_s$. Fix a non-zero $\eta\in\Gamma_s^G$, and let the $k$-subalgebra of 
$L$ generated by the ratios $\lambda/\eta\in L$ for all $\lambda\in\Gamma_s^G$ 
be the coordinate ring of an affine variety $W$ over $k$. If the pluricanonical map 
$\mathbb W\dashrightarrow\mathbb P(\Gamma_s^{\vee})$ is generically injective 
for some $s>0$ then $L=K(\mathbb W)$ is the function field of $W\times_kK$. 

If $\mathbb W$ is a smooth irreducible projective curve over $K$ of genus 0 then the group 
$G$ (i) naturally acts $L$-semilinearly and smoothly on the one-dimensional $L$-vector 
space $\mathrm{Der}_K(L)$, (ii) preserves the (three-dimensional) $K$-vector space 
$\Gamma(\mathbb W,\mathcal T_{\mathbb W|K})\subset\mathrm{Der}_K(L)$ of regular vector 
fields on $\mathbb W$, and acts on it $K$-semilinearly. Again by (\ref{generat-fin-dim}), 
$\Gamma(\mathbb W,\mathcal T_{\mathbb W|K})^G$ is a $k$-lattice, while the ratios of its 
elements generate the function field of conic $W$ over $k$ such that $L=K(\mathbb W)=K(W)$. 
\end{proof} 

\begin{example} \label{ex-triv-fdim} The condition (\ref{generat-fin-dim}) of 
Proposition~\ref{isotriv_gen-type_curves} is satisfied, e.g., in the following cases: a$)$ $K=F_{\Psi}$ 
and $G=\Sy_{\Psi}$ (see Theorem~\ref{H90Theorem1.2}), b$)$ $K$ is an algebraically closed extension of 
a field $k$ of infinite transcendence degree, and $G$ is the group of field automorphisms of $K$ 
identical on $k$ (see \cite[Proposition 5.4]{rep}). \end{example} 

The only class of curves not covered by Proposition~\ref{isotriv_gen-type_curves} are 
genus one curves. In this case, the result is less definitive, though it does not use 
the condition (\ref{generat-fin-dim}) of Proposition~\ref{isotriv_gen-type_curves}: 
\begin{proposition} \label{isotriv_elliptic_curves} Let $K$ be a field of characteristic $p\ge 0$, 
$L$ be the function field of an elliptic curve $\mathbb W$ over $K\ ($i.e. of a genus one curve with 
a $K$-rational point$)$.\footnote{Lemma~\ref{triv_torsors} below contains some sufficient 
conditions for the existence of a $K$-rational point on such a genus one curve.} Let $G$ be a permutation 
group, acting smoothly on $L$ and preserving $K$. Suppose that $H^1(G,K)=H^2(G,\mathbb Z/2\mathbb Z)=0$ 
and, in the case $j(\mathbb W)=0$, \begin{enumerate} \item $H^2(G,\{z\in K^{\times}~|~z^3=1\})=0$, 
\item $H^2(G,\mathbb Z/3\mathbb Z)=H^1(G,\mathbb Z/2\mathbb Z)=0$ if $p=3$, 
\item there are no finite index open subgroups in $G$ if $p=2$. \end{enumerate} 
Then $L^G=k(W)$ for an irreducible curve $W$ over $k:=K^G$ and $L=K(W)$. \end{proposition} 
\begin{proof} The Jacobian $E$ (punctured at $0$) of $\mathbb W$ is given by 
the Weierstra\ss\ equation $y^2+a_1xy+a_3y=x^3+a_2x^2+a_4x+a_6$ for some $a_i\in K$ (see 
\cite{Tate,Silverman}). The quintuple $(a_1,a_2,a_3,a_4,a_6)\in K^5$ is defined by $E$ uniquely 
modulo the action of the group 
\[H:=\{\xi:(x,y)\mapsto(x/c^2+d,y/c^3+ex+f)~|~c\in K^{\times},\ d,e,f\in K\}.\] 

The group $G$ fixes the isomorphism class of $E$, and our task is to show that the class of $E$ belongs 
to the image of the natural map $k^5\to(H\backslash K^5)^G$. Obviously, $G$ fixes the $j$-invariant 
$j(E):=(b_2^2-24b_4)^3/\Delta$ of $E$, i.e. $j(E)\in k$. Here $b_2:=a_1^2+4a_2$, $b_4:=2a_4+a_1a_3$, 
$b_6:=a_3^2+4a_6$, $b_8:=a_1^2a_6+4a_2a_6-a_1a_3a_4+a_2a_3^2-a_4^2$, 
$\Delta:=-b_2^2b_8-8b_4^3-27b_6^2+9b_2b_4b_6$.

If $p\neq 2,3$ then $E$ can be given by $y^2=x^3+a_4x+a_6$, while $H$ is reduced to the subgroup 
$\{(x,y)\mapsto(x/c^2,y/c^3)~|~c\in K^{\times}\}\cong K^{\times}$. The isomorphism class of $E$ is determined 
by the invariants $j(E)=1728\frac{4a_4^3}{4a_4^3+27a_6^2}$ and $\gamma(E)\in K^{\times}/K^{\times n_E}$, 
where (i) $n_E=2$ and $\gamma(E):=a_6/a_4\bmod{K^{\times 2}}$ if $a_4a_6\neq 0$; 
(ii) $n_E=4$ and $\gamma(E):=a_4\bmod{K^{\times 4}}$ if $j(E)=1728$; 
(iii) $n_E=6$ and $\gamma(E):=a_6\bmod{K^{\times 6}}$ if $j(E)=0$. 

For $p\in\{2,3\}$, $q\in\{p,p^2\}$ and $a\in k$, denote by $\mathcal A_{q,a}$ 
the cokernel of $\beta_{q,a}\colon K\xrightarrow{b\mapsto b^q-ab}K$. 

If $p=3$ then $E$ can be given either (i) by $y^2=x^3+a_2x^2+a_6$ with $H$ reduced to the subgroup 
$\{(x,y)\mapsto(x/c^2,y/c^3)~|~c\in K^{\times}\}\cong K^{\times}$, or 
(ii) by $y^2=x^3+a_4x+a_6$ (if $j(E)=0$) with $H$ reduced to the subgroup 
$\{(x,y)\mapsto(x/c^2+d,y/c^3)~|~c\in K^{\times},\ d\in K\}\cong K\rtimes_2K^{\times}$. 
Then, in (i), $a_2\bmod{K^{\times 2}}$ is well-defined, while $j(E)=a_2^6/(a_2^2a_4^2-a_2^3a_6-a_4^3)$ and 
$a_2\bmod{K^{\times 2}}$ determine the curve $E$ up to $K$-isomorphism. 
In (ii), $a_4\neq 0$ is well-defined modulo $K^{\times 4}$. It is shown below that our assumptions imply 
$a_4\in k^{\times}K^{\times 4}$, so we may and we do assume that $a_4\in k^{\times}$ is fixed. Then the 
equivalence class of $a_6$ is $\{\zeta^2a_6+c^3+a_4c~|~c\in K,\ \zeta\in\mu_4(K)\}$, so the image of 
$a_6$ in $\mathcal A_{3,a_4}$ determines a homomorphism $G\to\{\pm 1\}$, which should be trivial. 
This means that the class of $a_6$ is presented by an element of $\mathcal A_{3,a_4}^G$. 

If $p=2$ then $j(E):=a_1^{12}/\Delta$, where 
$\Delta=a_1^4(a_1^2a_6-a_4^2)+a_3(a_1^5a_4+a_1^4a_2a_3-a_3^3+a_1^3a_3^2)$. 
(i) If $j(E)\neq 0$, then this is equivalent to $y^2+xy=x^3+a_2x^2+a_6$. 
The group is $\{(x,y)\mapsto(x,y+cx)~|~c\in K\}\cong K$. 
Then $a_2$ is well-defined as an element of the group $\mathcal A_{2,1}$, while $j(E)=a_6^{-1}$ and 
$[a_2]\in\mathcal A_{2,1}$ determine the curve $E$ up to $K$-isomorphism. 
(ii) If $j(E)=0$ then $E$ can be given by an equation $y^2+a_3y=x^3+a_4x+a_6$. 
The group is $\{(x,y)\mapsto(c^2(x+d^2),c^3(y+dx+e))~|~c\in K^{\times},\ d,e\in K\}$. 
Then $a_3\neq 0$ is well-defined modulo $K^{\times 3}$. It is shown below that our assumptions imply 
$a_3\in k^{\times}K^{\times 3}$, so we may and we do assume that $a_3\in k^{\times}$ is fixed. 
Then the equivalence class of $a_4$ is $\{\zeta a_4+c^4+a_3c~|~c\in K,\ \zeta\in\mu_3(K)\}$. 
The image of $a_4$ in $\mathcal A_{4,a_3}$ determines a homomorphism 
$G\to\mu_3(K)$, which should be trivial. This means that the class of $a_4$ is presented by an element 
of $\mathcal A_{4,a_3}^G$. As shown below, $k\to\mathcal A_{3,a_4}^G$ is surjective, so we may and 
we do assume that $a_4\in k$ is fixed. Now the equivalence class of $a_6$ becomes 
$\{a_6+a_4d^2+a_3d^3+e^2+a_3e~|~e\in K,\ 1+(d^4+a_3d)/a_4\in\mu_3(K)\}$, so the $G$-orbit of the class 
of $a_6$ in $\mathcal A_{2,a_3}$ is of order $\le 12$. Assuming that there are no open subgroups in $G$ 
of finite index, we get that the class of $a_6$ is presented by an element of $\mathcal A_{2,a_3}^G$. 
As shown below, $k\to\mathcal A_{2,a_3}^G$ is surjective, so we may and we do assume that $a_6\in k$.

As (i) the group $G$ fixes these invariants, (ii) $H^2(G,\mu_6)=0$, 
(iii) $\mathrm{Pic}_K(G)=0$, Lemma~\ref{Pic-torsion} shows that all these invariants belong to 
$k^{\times}K^{\times n}/K^{\times n}$, i.e. $E$ is the base change to $K$ of an elliptic curve over $k$. 
The exact sequence $0\to\mathbb F_q\cap K\to K\xrightarrow{\beta_{q,a}}K\to\mathcal A_{q,a}\to 0$ 
gives exact sequences $k\to\mathcal A_{q,a}^G\to H^1(G,\mathrm{Im}(\beta_{q,a}))$ 
and $H^1(G,K)\to H^1(G,\mathrm{Im}(\beta_{q,a}))\to H^2(G,\mathbb F_q\cap K)$. But 
$H^1(G,K)=0$, and thus, $k\to\mathcal A_{q,a}^G$ is surjective whenever $H^2(G,\mathbb F_q\cap K)=0$. 
For $q=4$, the latter vanishing holds if and only if $H^2(G,\mathbb F_2)=0$. \end{proof} 

\begin{example} By \cite[Theorem 3.5, Corollary 3.7]{EvansHewitt}, all assumptions of 
Proposition~\ref{isotriv_elliptic_curves} on $G$ are satisfied for $G=\Sy_{\Psi}$. \end{example} 

\begin{lemma} \label{triv_torsors} Let $G$ be a permutation group, and $L|K$ be a smooth $G$-field extension 
of characteristic $0$. Suppose that \begin{enumerate} \item the field $k:=K^G$ is algebraically closed, 
\item $L$ is the function field of a torsor $\mathbb W$ under $A\times_kK$ for an abelian variety $A$ over $k$; 
\item $\mathrm{Pic}_K(G)=H^2(G,\mathbb Z/\ell\mathbb Z)=0$ for all prime $\ell$, 
\item $K$ is unirational over $k$. \end{enumerate} Then $\mathbb W\cong A\times_kK$. \end{lemma} 
\begin{proof} Let $\overline{K}$ be an algebraic closure of $K$. Let 
$H^1(K,-):=H^1(\mathrm{Gal}(\overline{K}|K),-)$ denote the Galois cohomology group. 
The torsors under $A\times_kK$ are classified by $H^1(K,A(\overline{K}))$, 
so it suffices to show the vanishing of $H^1(K,A(\overline{K}))^G$. 
As $H^1(K,A(\overline{K}))$ is a torsion group, it suffices to check the vanishing of 
${}_nH^1(K,A(\overline{K}))^G$ for all $n>1$. Since $A(K)=A(k)$ is $n$-divisible, the sequence 
$0\to{}_nA(\overline{K})\to A(\overline{K})\xrightarrow{\times n}A(\overline{K})\to 0$ 
is exact and gives an isomorphism of Galois cohomology 
$H^1(K,{}_nA(\overline{K}))\xrightarrow{\sim}{}_nH^1(K,A(\overline{K}))$, and thus, we get 
an isomorphism $H^1(K,{}_nA(\overline{K}))^G\xrightarrow{\sim}{}_nH^1(K,A(\overline{K}))^G$. 

As ${}_nA(\overline{K})={}_nA(k)$ is a trivial Galois module, 
$H^1(K,{}_nA(\overline{K}))=\Hom_{\mathrm{cont}}(\mathrm{Gal}(\overline{K}|K),{}_nA(k))=
(K^{\times}/K^{\times n})\otimes_{\mathbb Z}\Hom_{\mathbb Z}(\mu_n,A(k))$. 

By Lemma~\ref{Pic-torsion}, the natural map $k^{\times}\to(K^{\times}/K^{\times n})^G$ 
is surjective, but $k^{\times}$ is $n$-divisible, so $(K^{\times}/K^{\times n})^G=0$, 
and therefore, ${}_nH^1(K,A(\overline{K}))^G=0$. \end{proof} 

\begin{example} \label{exa-triv_torsors} All assumptions of Lemma~\ref{triv_torsors} are satisfied, 
e.g., for $G=\Sy_{\Psi}$ and $K=F_{\Psi}$, where $F$ is unirational over $k$. \end{example}

\subsection{Isotrivial finitely generated \texorpdfstring{$\Sy_{\Psi}$}{}-extensions of 
\texorpdfstring{$F_{\Psi}$}{}} 
\begin{proposition} \label{isotriv_fin_gen-ext} Let $F|k$ be a regular field extension, 
$Y$ be a geometrically irreducible $k$-variety, and 
$\Sy_{\Psi}\xrightarrow{\tau\mapsto\tau^{\iota}}\Aut_{\mathrm{field}}(K|k)$ be a smooth 
$\Sy_{\Psi}$-action on the field $K:=F_{\Psi}(Y)$ extending the $\Sy_{\Psi}$-action on $F_{\Psi}$. 
Then $K^{\Sy_{\Psi}}\cong k(Y')$ for a geometrically irreducible $k$-variety $Y'$ such that 
$Y_{k'}$ and $Y'_{k'}$ are birational for a finite extension $k'|k$. \end{proposition} 
\begin{proof} As the pointwise stabilizer of $k(Y)$ is open, it contains $\Sy_{\Psi|J}$ for a finite subset 
$J\subset\Psi$. For each $\tau\in\Sy_{\Psi}$, $\tau^{\iota}(k(Y))$ is fixed by $\Sy_{\Psi|\tau(J)}$, 
so $\tau^{\iota}(k(Y))\subseteq K^{\Sy_{\Psi|\tau(J)}}\subseteq K^{\Sy_{\Psi|J\cup\tau(J)}}=
F_{J\cup\tau(J)}(Y)$. 

Let $\eta\in\Sy_{\Psi}$ be an involution such that $J\cap\eta(J)=\varnothing$, and $F'\subseteq F$ be 
a finitely generated field extension of $k$ such that $\eta^{\iota}(k(Y))\subseteq F'_{J\cup\eta(J)}(Y)$. 
As $\eta$ and $\Sy_{\Psi|J}$ generate $\Sy_{\Psi}$, the subfield $F'_{\Psi}(Y)\subseteq K$ is 
$\Sy_{\Psi}$-invariant, so we may further assume that 
$F=F'$. As $F_{J\cup\tau(J)}\tau^{\iota}(k(Y))\subseteq F_{J\cup\tau(J)}(Y)$ and 
$F_{J\cup\tau^{-1}(J)}(\tau^{-1})^{\iota}(k(Y))\subseteq F_{J\cup\tau^{-1}(J)}(Y)$ (or equivalently, 
$F_{J\cup\tau(J)}(Y)\subseteq F_{J\cup\tau(J)}\tau^{\iota}(k(Y))$), 
we get $F_{J\cup\tau(J)}\tau^{\iota}(k(Y))=F_{J\cup\tau(J)}(Y)$. 

This gives rise to a field automorphism 
$\alpha_{\tau}\in\Aut(F_{J\cup\tau(J)}(Y)|F_{J\cup\tau(J)})$ extending the $F_{J\cup\tau(J)}$-algebra 
homomorphism $F_{J\cup\tau(J)}\otimes_kk(Y)\xrightarrow{id\cdot\tau^{\iota}}F_{J\cup\tau(J)}(Y)$. 

For each subset $I\subset\Psi$, the group $\Aut(F_I(Y)|F_I)$ is naturally embedded into $\Aut(K|F_{\Psi})$. 
Denote by $\Sy_{\Psi}\xrightarrow{\tau\mapsto\tau_Y}\Aut_{\mathrm{field}}(K|k(Y))$ the embedding given 
by the standard $\Sy_{\Psi}$-action on $F_{\Psi}$ and the trivial $\Sy_{\Psi}$-action on $k(Y)$. 
Clearly, $\sigma^{\iota}=\alpha_{\sigma}\sigma_Y$ for any $\sigma\in\Sy_{\Psi}$, and $\Sy_{\Psi}$ 
acts on $\Aut(K|F_{\Psi})$ by conjugation: $\sigma\colon\alpha\mapsto\alpha^{\sigma}:=
\sigma_Y\alpha\sigma_Y^{-1}$ for all $\sigma\in\Sy_{\Psi}$ and $\alpha\in\Aut(K|F_{\Psi})$. 

Then $\alpha_{\sigma\tau}(\sigma\tau)_Y=(\sigma\tau)^{\iota}=\sigma^{\iota}\tau^{\iota}=
\alpha_{\sigma}\sigma_Y\alpha_{\tau}\tau_Y=\alpha_{\sigma}\alpha_{\tau}^{\sigma}(\sigma\tau)_Y$ 
for all $\sigma,\tau\in\Sy_{\Psi}$, so $\alpha_{\sigma}\alpha_{\tau}^{\sigma}=\alpha_{\sigma\tau}$. 

Fix some $\sigma\in\Sy_{\Psi}$ such that $|J\cup\sigma(J)\cup\sigma\eta(J)|=3|J|$. As the field 
$F_{J\cup\eta(J)}(Y)$ is finitely generated over $k$ and $k$ is algebraically closed in $F_{J\cup\eta(J)}(Y)$, 
there is a finite Galois extension $k'|k$ and a place $F_J\dashrightarrow k'$ 
such that the compositions of $\alpha_{\sigma}$ and of $\alpha_{\sigma\eta}$ with the induced place 
$s\colon F_{J\cup\sigma(J)\cup\sigma\eta(J)}(Y)\dashrightarrow F_{\sigma(J)\cup\sigma\eta(J)}(Y)\otimes_kk'$ 
are well-defined. Then the cocycle condition implies $s(\alpha_{\sigma})\alpha_{\eta}^{\sigma}=s(\alpha_{\sigma\eta})$, 
where $s(\alpha_{\xi})\in\Aut(k'F_{\xi(J)}(Y)|k'F_{\xi(J)})$ for both $\xi\in\{\sigma,\sigma\eta\}$. Equivalently, 
$\alpha_{\eta}=s_{\sigma}^{-1}s_{\sigma\eta}^{\eta}$, where 
$s_{\xi}:=s(\alpha_{\xi})^{\xi^{-1}}\in\Aut(k'F_J(Y)|k'F_J)$ for both $\xi\in\{\sigma,\sigma\eta\}$. 

Replacing the subfield $k'(Y)$ of the field $K\otimes_kk'$ by $s_{\sigma\eta}^{-1}(k'(Y))$, we may assume 
further that $\alpha_{\eta}\in\Aut(k'F_J(Y)|k'F_J)$ and still $\alpha_{\sigma}=1$ for all 
$\sigma\in\Sy_{\Psi|J}$. The 1-cocycle relation $\alpha_{\eta}\alpha_{\eta}^{\eta}=1$ implies that 
$\alpha_{\eta}=(\alpha_{\eta}^{\eta})^{-1}\in\Aut(k'F_J(Y)|k'F_J)\cap\Aut(k'F_{\eta(J)}(Y)|k'F_{\eta(J)})
=\Aut(k'(Y)|k')$. Therefore, the map $\Sy_{\Psi}\to\Aut(k'(Y)|k')$, $\xi\mapsto\alpha_{\xi}$, is 
a homomorphism with open kernel, which is trivial by the simplicity of $\Sy_{\Psi}$. This means that 
$(K\otimes_kk')^{\iota(\Sy_{\Psi})}=k'(Y)$. As the $\mathrm{Gal}(k'|k)$-action on $K\otimes_kk'$ commutes 
with the $\Sy_{\Psi}$-action, $K^{\Sy_{\Psi}}=((K\otimes_kk')^{\mathrm{Gal}(k'|k)})^{\Sy_{\Psi}}
=((K\otimes_kk')^{\Sy_{\Psi}})^{\mathrm{Gal}(k'|k)}=k'(Y)^{\mathrm{Gal}(k'|k)}=:k(Y')$. 
Then $K$ is a finite $\Sy_{\Psi}$-extension of $F_{\Psi}(Y')$, but 
$K\cong F_{\Psi}(Y')^{\oplus[K:F_{\Psi}(Y')]}$ in $\Sm_{F_{\Psi}(Y')}(\Sy_{\Psi})$, so 
$k(Y')=K^{\Sy_{\Psi}}\cong k(Y')^{\oplus[K:F_{\Psi}(Y')]}$, and thus, $K=F_{\Psi}(Y')$. \end{proof} 

\begin{remark} The $k$-varieties $Y$ and $Y'$ in Proposition~\ref{isotriv_fin_gen-ext} need not be 
birational. E.g. a plane quadric with no rational points $Y$ and $Y':=\mathbb P^1_k$ are not birational, 
while $Y_F$ has a rational point if $F\cong k(Y)$, so $F_{\Psi}(Y)=F_{\Psi}(Y')$. \end{remark}

\appendix 
\section{The spectrum of \texorpdfstring{$\Sm_{F_{\Psi}}(\Sy_{\Psi})$}{} 
(Addendum to \texorpdfstring{\cite{H90}}{[5]})} \label{Addendum-to-H90}

The following result differs from \cite[Theorem 1.2]{H90} only in that there is no restriction on the 
transcendence degree of $F|k$. The same proof goes through but with a little modification: 
certain fields should be embedded into a field of Hahn $\Gamma$-power series\footnote{Let $k$ be a field, 
and $\Gamma$ be a totally ordered abelian group. The field $k((\Gamma))$ of {\sl Hahn power series} over 
$k$ with value group $\Gamma$ is the set of formal expressions of the form $\sum_{s\in\Gamma}a_s\cdot s$, 
where $a_s\in k$ and the set $\{s\in\Gamma\ |\ s<\gamma,\ a_s\neq 0\}$ is finite for all 
$\gamma\in\Gamma$. The addition and multiplication are defined in the obvious way.} for any {\sl 
sufficiently large} totally ordered divisible group $\Gamma$, instead of just Hahn $\mathbb Q$-power series. 

\begin{theorem} \label{H90Theorem1.2} Let $\Psi$ be an infinite set, and $F|k$ be a non-trivial 
regular field extension. Then {\rm (i)} the object $F_{\Psi}$ is an injective cogenerator of the category 
$\Sm_K(\Sy_{\Psi})$ for any $\Sy_{\Psi}$-invariant subfield $K\subseteq F_{\Psi};$ {\rm (ii)} the objects 
$F_{\Psi}\langle\binom{\Psi}{s}\rangle\cong\bigwedge^s_{F_{\Psi}}F_{\Psi}\langle\Psi\rangle$ for all 
integer $s\ge 0$, where $\binom{\Psi}{s}$ denotes the set of all subsets of $\Psi$ of cardinality $s$, 
present all isomorphism classes of indecomposable injectives in $\Sm_{F_{\Psi}}(\Sy_{\Psi})$. \end{theorem} 

\begin{remark} Let $G$ be a permutation group, $K$ be a smooth 
$G$-field, $A\subset K$ be a $G$-invariant subring, and $V$ be an injective object of $\Sm_K(G)$. 
Then $V$ is an injective object of $\Sm_A(G)$, since the functor $\Hom_{\Sm_A(G)}(-,V)$ is the composition 
of exact functors $\Sm_A(G)\xrightarrow{K\otimes_A(-)}\Sm_K(G)$ and $\Hom_{\Sm_K(G)}(-,V)$. Evidently, 
any cogenerator of $\Sm_K(G)$ is simultaneously an injective cogenerator of $\Sm_A(G)$. \end{remark} 

\begin{proof} (i) This is \cite[Theorem 3.10]{H90}, but the proof there depends on the injectivity of 
$F_{\Psi}$ (\cite[Proposition 3.8]{H90}). The argument of \cite[Proposition 3.8]{H90} uses the cardinality 
restriction on transcendence degree of $F|k$ to construct, for any finite $J\subset\Psi$, a morphism of 
$F_{\Psi\smallsetminus J}\langle\Sy_{\Psi|J}\rangle$-modules $\xi\colon F_{\Psi}\to F_{\Psi\smallsetminus J}$ 
identical on $F_{\Psi\smallsetminus J}$. We consider $F_{\Psi}$ as the fraction field of the algebra 
$F_{\Psi\smallsetminus J}\otimes_kF_J$ and embed both, the algebra and the field, into a field of series with 
coefficients in $F_{\Psi\smallsetminus J}\otimes_k\overline{k}$ for an algebraic closure $\overline{k}$ of $k$.  

Fix a totally ordered $\mathbb Q$-vector space $\Gamma$, such that transcendence degree over $k$ of 
the field $k((\Gamma))$ of Hahn power series is at least that of $F|k$. 
By \cite{MacLane}, there is a field embedding $F_J\hookrightarrow\overline{k}((\Gamma))$ 
over $k$, so the $\Sy_{\Psi|J}$-field $F_{\Psi}$ becomes a subfield of 
$(F_{\Psi\smallsetminus J}\otimes_k\overline{k})((\Gamma))$ with $\Sy_{\Psi|J}$ acting on 
the coefficients. Define $\widetilde{\xi}\colon F_{\Psi}\to F_{\Psi\smallsetminus J}\otimes_k\overline{k}$ 
as the `constant term'\ of the Hahn power series expression: $\sum_{s\in\Gamma}a_s\cdot s\mapsto a_0$. 
Fix a $k$-linear functional $\nu\colon\overline{k}\to k$ identical on $k$. Finally, we define 
$\xi\colon F_{\Psi}\to F_{\Psi\smallsetminus J}$ as $(id_{F_{\Psi\smallsetminus J}}\cdot\nu)\circ\widetilde{\xi}$. 

(ii) Let us check that $F_{\Psi}\langle\binom{\Psi}{s}\rangle$ is injective. 

Let $K\subset F_{\Psi}(\Psi)$ be the subfield generated over $F_{\Psi}$ by squares of 
the elements of $\Psi$. There is an isomorphism of $F_{\Psi}\langle\Psi\rangle$-modules 
$\bigoplus_{s\ge 0}K\langle\binom{\Psi}{s}\rangle\xrightarrow{\sim}F_{\Psi}(\Psi)$, 
$[S]\mapsto\prod_{t\in S}t$, so each $K\langle\binom{\Psi}{s}\rangle$ is isomorphic to a direct 
summand of the object $F_{\Psi}(\Psi)$ of $\Sm_K(\Sy_{\Psi})$. 

As $K$ and $F_{\Psi}(\Psi)$ are isomorphic $\Sy_{\Psi}$-field extensions of $F_{\Psi}$ (under 
$t^2\mapsto t$ for all $t\in\Psi$), they are isomorphic as objects of $\Sm_{F_{\Psi}}(\Sy_{\Psi})$. 

By (i), $F_{\Psi}(\Psi)=F(T)_{\Psi}$ is injective in $\Sm_{F_{\Psi}}(\Sy_{\Psi})$, so each 
$K\langle\binom{\Psi}{s}\rangle$ is an injective object of $\Sm_{F_{\Psi}}(\Sy_{\Psi})$. As 
$F_{\Psi}$ is injective as well, the inclusion $F_{\Psi}\hookrightarrow K$ admits a splitting 
$K\xrightarrow{\pi}F_{\Psi}$. Then the inclusion 
$F_{\Psi}\langle\binom{\Psi}{s}\rangle\hookrightarrow K\langle\binom{\Psi}{s}\rangle$ splits as well: 
$\sum_ia_i[S_i]\mapsto\sum_i\pi(a_i)[S_i]$, and thus, $F_{\Psi}\langle\binom{\Psi}{s}\rangle$ 
is an injective object of $\Sm_{F_{\Psi}}(\Sy_{\Psi})$. 

It remains to show that any smooth finitely generated $F_{\Psi}\langle\Sy_{\Psi}\rangle$-module 
can be embedded into a direct sum of $F_{\Psi}\langle\binom{\Psi}{s}\rangle$ for several integer 
$s\ge 0$. But this is already shown in \cite[Proof of Theorem 1.2, pp.2331--2332]{H90}. \end{proof} 

\section{Examples of algebraic groups acting on a field; fixed fields} \label{group-actions-on-fields}
Let $L|k$ be a field extension, and $H$ be an algebraic $k$-group. An $H$-{\sl action} on $L$ is a 
rational $k$-map $H\times_kL\xrightarrow{\tau}\mathop{\mathrm{Spec}}(L)$ which is a) associative, i.e. 
the compositions $H\times_kH\times_kL\rightrightarrows H\times_kL\xrightarrow{\tau}\mathop{\mathrm{Spec}}(L)$, 
where the parallel arrows are $(h,h',a)\mapsto(hh',a)$ and $(h,h',a)\mapsto(h,\tau(h',a))$, coincide, and b) 
such that the endomorphism $(h,a)\mapsto(h,\tau(h,a))$ of the $k$-scheme $H\times_kL$ is invertible. 
If $H=\mathop{\mathrm{Spec}}(A)$ is affine, where $A$ is a Hopf $k$-algebra with a 
comultiplication $A\xrightarrow{\Delta}A\otimes_kA$, this is equivalent to the bijectivity 
of $id_A\cdot\tau^*$ on $\widetilde{A\otimes_kL}$ and to the coincidence of the compositions 
$L\xrightarrow{\tau^*}\widetilde{A\otimes_kL}\rightrightarrows\widetilde{A\otimes_kA\otimes_kL}$, 
where the parallel arrows are $\Delta\otimes id_L$ and $id_A\otimes\tau^*$, while $\widetilde{\ \ }$ 
denotes total ring of fractions (i.e. invertion of all elements that are not zero divisors). 

By definition, the {\sl fixed field} $L^H$ of $H$ in $L$ is the coordinate field of the coequalizer 
of the morphisms $\tau,\mathrm{pr}_L\colon H\times_kL\rightrightarrows\mathop{\mathrm{Spec}}(L)$ (or $L^H$ 
is the equalizer of embeddings $\tau^*,1\otimes id_L\colon L\rightrightarrows\widetilde{A\otimes_kL}$ 
if $H=\mathop{\mathrm{Spec}}(A)$). 

Obviously, (i) any subgroup $H_0$ of the group $H(k)$ (of $k$-rational points of $H$) fixes in $L$ the same 
subfield $\{a\in L~|~a^g=a\mbox{ for all }g\in H_0\}$ as the Zariski closure of $H_0$ in $H$; (ii) the 
equalizer of any set of field embeddings of $L$ into a reduced ring is a relatively perfect subfield in $L$. 

But here is an example of a {\sl non-perfect} subfield fixed by an {\sl algebraic} $k$-group: 
\begin{example}[Fixed fields of $\alpha_{p^n,k}$] \label{alpha_p-invar} Let $F|k$ 
be a regular field extension of characteristic $p>0$ admitting an element $X\in F\smallsetminus k$ such that 
$F|k(X)$ is {\sl algebraic} separable, and $n\ge 1$ be an integer. Then, for each $\lambda\in kF^{p^n}$, 
there is a unique $k[B]$-algebra endomorphism $\xi_{\lambda}$ of $F_{\Psi}[B]/(B^{p^n})$ such that 
(i) $\xi_{\lambda}$ is identical modulo $(B)$, (ii) $u:=X(u)\mapsto u+\lambda(u)B$ for all $u\in\Psi$. 

The endomorphism $\xi_{\lambda}$ is invertible, $\Sy_{\Psi}$-equivariant and $(kF^{p^n})_{\Psi}$-linear. 

The automorphism $\xi_{\lambda}$ can be considered as an action $F_{\Psi}\to F_{\Psi}[B]/(B^{p^n})$ on 
the {\sl field} $F_{\Psi}$ of the infinitesimal subgroup $\alpha_{p^n,k}:=\mathrm{Spec}(k[B]/(B^{p^n}))$ 
of the additive group $\mathbb G_{\mathrm{a},k}:=\mathop{\mathrm{Spec}}(k[B])$. 

For each subset $\Lambda\subset kF^{p^n}$, denote by $K_{\Lambda}$ the subfield of $F_{\Psi}$ 
fixed by $\xi_{\lambda}$ for all $\lambda\in\Lambda$. 

Then $kF^{p^n}\xrightarrow{\lambda\mapsto\xi_{\lambda}}\Aut_{k[B]\mbox{-alg}}(F_{\Psi}[B]/(B^{p^n}))$ is a group 
homomorphism, so it is clear that $K_{\Lambda}$ depends only on the $k$-linear span of $\Lambda$. For this 
reason, we assume further that $\Lambda$ is a $d$-dimensional $k$-vector subspace. Set 
$A:=\mathrm{Sym}_k^{\bullet}\Lambda^{\vee}/((\Lambda^{\vee})^{p^n})$, where $\Lambda^{\vee}:=\Hom_k(\Lambda,k)$. 
For each finite-dimensional 
$\Lambda$, denote by $\varphi_{\Lambda}\in\Lambda\otimes_k\Lambda^{\vee}\subset F\otimes_kA$ the element 
corresponding to the identity endomorphism of $\Lambda$. For each $u\in\Psi$, the isomorphism 
$(u):F\xrightarrow{\sim}F_u$ sends $\varphi_{\Lambda}$ to an element 
$\varphi_{\Lambda}(u)\in(kF^{p^n})_u\otimes_k\Lambda^{\vee}\subset F_u\otimes_kA\subset F_{\Psi}\otimes_kA$. 
This gives rise to an action $F_{\Psi}\xrightarrow{X(u)\mapsto X(u)+\varphi(u)}F_{\Psi}\otimes_kA$ on 
$F_{\Psi}$ of the $\dim_k\Lambda$-th cartesian power of $\alpha_{p^n,k}$. 

If $\{\lambda_1,\dots,\lambda_d\}$ is a basis of $\Lambda$ then $[F_{\Psi}:K_{\Lambda}]=p^{nd}$, and 
$K_{\Lambda}$ is generated over $kF_{\Psi}^{p^n}$ by the elements 
$\sum_{\sigma\in\Sy_{\{0,\dots,d\}}}\mathrm{sgn}(\sigma)\lambda_1(u_{\sigma(1)})\cdots\lambda_d(u_{\sigma(d)})
u_{\sigma(0)}$ for all $u_0,\dots,u_d\in\Psi$. 
\begin{proof} The existence and uniqueness of $\xi_{\lambda}$ is obvious. The (co)associativity of $\xi_{\lambda}$ 
is clear: $u\mapsto u+\lambda(u)B\mapsto u+\lambda(u)(B+B')=u+\lambda(u)B'+\lambda(u+\lambda(u)B')B$. 
The (co)commutativity: 
$u\mapsto u+\lambda(u)B\mapsto u+\lambda'(u)B+\lambda(u+\lambda'(u)B)B=u+(\lambda(u)+\lambda'(u))B$. 
\end{proof} \end{example}

\vspace{5mm}

\noindent 
{\sl Acknowledgements.} {\small The study has been funded within the framework 
of the HSE University Basic Research Program. 
I am grateful to Dmitry Kaledin and Alexander Kuznetsov for very helpful suggestions and to Tao Qin 
for useful discussions. Comments of Anton Fonarev, Sergey Gorchinskiy and of the referee 
led to significant improvements of the original text.}

\end{document}